\documentclass[a4paper, 10pt, reqno]{amsart}
\usepackage{amsmath, amscd}
\usepackage{amssymb, color}
\usepackage[T1]{fontenc}
\usepackage{charter}
\usepackage[colorlinks]{hyperref}
\definecolor{linkblue}{RGB}{1,1,190}
\definecolor{citered}{RGB}{190,1,1}
\hypersetup{
	linkcolor=linkblue,
	urlcolor=linkblue,
	citecolor=citered
}
\usepackage{comment}

\makeatletter
\def\namedlabel#1#2{\begingroup
    #2%
    \def\@currentlabel{#2}%
    \phantomsection\label{#1}\endgroup
}
\makeatother

\def\doi#1{{\small\href{https://doi.org/#1}{\path{doi:#1}}}}
\def\arxiv#1{{\small\href{http://www.arxiv.org/abs/#1}{\path{arXiv:#1}}}}
\def\url#1{{\small\href{#1}{\path{#1}}}}

\theoremstyle{plain}
\newtheorem{theorem}{\bf Theorem}[section]
\newtheorem{proposition}[theorem]{\bf Proposition}
\newtheorem{lemma}[theorem]{\bf Lemma}
\newtheorem{corollary}[theorem]{\bf Corollary}

\theoremstyle{definition}

\newtheorem{remark}[theorem]{\bf Remark}

\newcommand{\N}{\mathbb N}

\newcommand{\bdot}{\boldsymbol{\cdot}}

 \DeclareMathOperator{\ord}{ord}
\DeclareMathOperator{\lcm}{lcm} \DeclareMathOperator{\Min}{Min}
 
\DeclareMathOperator{\spec}{spec} \DeclareMathOperator{\supp}{supp}

\numberwithin{equation}{section}

\newcommand{\pronor}{%
  \mathrel{\ooalign{$\lneq$\cr\raise.22ex\hbox{$\lhd$}\cr}}}

\begin{document}

\title{On algebraic and arithmetic properties of monoids of product-$K$ sequences}

\author{Jun Seok Oh}
\address{Department of Mathematics Education, Jeju National University, Jeju 63243, Republic of Korea}
\email{junseok.oh@jejunu.ac.kr}

\author{Doniyor Yazdonov}
\address{Department of Mathematics and Scientific Computing, University of Graz, NAWI Graz, 8010 Graz, Austria}
\email{doniyor.yazdonov@uni-graz.at}


\thanks{J.S. Oh was supported by the National Research Foundation of Korea (NRF) grant funded by the Korea government (MSIT) (NRF-2021R1G1A1095553). D. Yazdonov was supported by FWF, Project Number DOC 183-N}

\begin{abstract}
Let $G$ be a group and $K$ be a normal subgroup of $G$.
A sequence over $G$ is a finite collection of terms from $G$, where repetition is allowed, and the order is disregarded.
A product-$K$ sequence is a sequence whose terms can be ordered such that their product in $G$ belongs to $K$.
The set $\mathcal B_K (G)$ of all product-$K$ sequences over $G$ forms a monoid, called the monoid of product-$K$ sequences, under the operation of sequence concatenation.
In this paper, we investigate the algebraic and arithmetic properties of the monoid $\mathcal B_K (G)$.
Among our main results, we provide precise characterizations of when $\mathcal B_K (G)$ satisfies key properties, namely being (transfer) Krull, a (seminormal) C-monoid, or (half-)factorial.
After that, we study the Isomorphism Problem for monoids of product-$K$ sequences.
Our results generalize existing frameworks, making them applicable to both the classical abelian and the more recently developed non-abelian settings.
\end{abstract}
\maketitle

% % % % % % % % % % % % % % % % % % % % % % % % % % % % % % % % % % % % % % % % % % % % % % % % % % % % % %
% % %             % % % % % % %               % % % % % % %               % % % % % % %               % % %
% % %             % % % % % % %               % % % % % % %               % % % % % % %               % % %
% % % % % % % % % % % % % % % % % % % % % % % % % % % % % % % % % % % % % % % % % % % % % % % % % % % % % %

\section{Introduction} \label{1}
\medskip

The study of non-unique factorization properties in algebraic structures is a central theme in factorization theory, and it is, to a large extent, investigated by making use of methods from arithmetic combinatorics.
Thus, to systematically understand the phenomena of non-unique factorizations in (transfer) Krull monoids and C-monoids, combinatorial models play a crucial role.
Among these, the monoid of product-one sequences over a finite group stands out as a fundamental combinatorial object that has been a central subject of study in arithmetic combinatorics and has deep connections to various branches of mathematics, most notably factorization theory and invariant theory.
Roughly speaking, in invariant theory, the (separating) Noether number--a fundamental invariant of polynomial invariant rings--is strictly governed by the Davenport constant of the underlying group.
For a sample of recent developments in these directions, we refer the reader to \cite{Cz-Do-Ge16,Cz-Do-Sz18,Do-Sc24,Do-Sc25,Sc25a,Sc25b,Sc-Zha-Zh26,Sc-Zha-Zh27,Hu26} for invariant theory and to \cite{Ge-Gr13,Gr13b,Oh19,Oh20,Oh-Zh20a,Ge-Gr-Oh-Zh22,Qu-Li-Tee25,Ge-Oh25,Oh26,Oh26b} for factorization theory.

To be precise, let $G$ be a multiplicatively written (not necessarily finite) group and $K$ be a normal subgroup of $G$.
By a sequence, we mean a finite collection of terms from $G$, where repetition is allowed, and the order is disregarded.
In an algebraic context, a sequence can be viewed as an element in the free abelian monoid $\mathcal F (G)$ with basis $G$.
A sequence is said to be product-$K$ if its terms can be ordered such that their product in $G$ belongs to $K$.
The set $\mathcal B_K (G)$ of product-$K$ sequences over $G$ forms a submonoid of $\mathcal F (G)$, and is called the monoid of product-$K$ sequences.

When $G$ is abelian, additive notation is predominantly used, and product-$K$ sequences are referred to as $K$-sum sequences.
The corresponding theory has been extensively developed.
While the case of trivial $K$ dates back to its initial study in 1961, the more general notion of $K$-sum sequences with non-trivial $K$ was introduced by Halter-Koch (see \cite{HK92}) in the study of the arithmetic of residue classes.
Later, in \cite{Ba-Ho09}, the authors studied atoms in the monoid $\mathcal B_K (G)$.
Recently, the monoid $\mathcal B_K (G)$ has served as a crucial tool for the Characterization Problem for an abelian group of rank 2 (see \cite{Ge-Sc19b}).
The Characterization Problem--the standing conjecture in factorization theory asking whether the non-unique factorization phenomena of the monoid of (weighted) product-one sequences uniquely determine the given finite group--has attracted wide attention and received an affirmative answer for almost all abelian groups (see \cite{Ge-Zh15,Zh18a,Ge-Sc19b}).
Exploration of this problem has also begun in the non-abelian setting, where it has been verified for certain non-abelian groups (see \cite{Oh19,Ge-Gr-Oh-Zh22}).
In a very recent year, in \cite{Ge-Oh25}, the authors studied a necessary condition for an affirmative answer to the Characterization Problem for general groups, known as the Isomorphism Problem.
The Isomorphism Problem asks whether the monoid structure uniquely determines the structure of the underlying group, and it has been verified for torsion groups (see \cite[Theorem 1.1]{Ge-Oh25}).
However, the Isomorphism Problem for $\mathcal B_K(G)$ fails to hold even in the abelian case.
More precisely, if $G_1$ and $G_2$ are abelian groups of order at least 3 with normal subgroups $K_1$ and $K_2$ respectively, and the monoids $\mathcal B_{K_1} (G_1)$ and $\mathcal B_{K_2} (G_2)$ are isomorphic as monoids, then we can only conclude that $G_1 / K_1$ and $G_2 / K_2$ are isomorphic as abelian groups.
This indicates that, in the non-abelian setting, the monoid $\mathcal B_K (G)$ is a significantly more intricate object.
Nevertheless, just as it provided crucial leverage in the abelian setting, developing the theory of $\mathcal B_K (G)$ in the non-abelian setting is strongly anticipated to open new avenues for tackling the Characterization Problem for non-abelian groups.

Independently of its applications to the Characterization Problem, the monoid $\mathcal B_K (G)$ is of particular interest for certain normal subgroups $K$.
If $K$ is trivial, then $\mathcal B_K (G)$ is the monoid of product-one sequences, whereas if $K$ is the commutator subgroup of a torsion group $G$, then since $G/K$ is abelian, $\mathcal B_K (G)$ is the complete integral closure of the monoid of product-one sequences (see \cite[Lemma 3.4]{Fa-Zh23}).
This provides further motivation for studying the monoid $\mathcal B_K (G)$.

In this paper, we investigate the algebraic and arithmetic properties of analogous monoids in the non-abelian setting.
Among our main contributions, we provide precise algebraic characterizations for the monoid $\mathcal B_K (G)$.
Specifically, in the case where $G/K$ is a torsion group, we establish the exact conditions under which the monoid $\mathcal B_K (G)$ is (transfer) Krull, (half-)factorial, or a (seminormal) C-monoid (see Theorems~\ref{thm:struc}, \ref{thm:transfer}, and Corollary \ref{coro:seminormal}).
Furthermore, we show that arithmetic properties for the monoid $\mathcal B_K (G)$ can be effectively studied via product-one sequences over a suitable group.
As a major application of this, we deduce that the union of sets of lengths for the monoid $\mathcal B_K (G)$ forms an interval (see Theorem~\ref{thm:transfer}).
Finally, we investigate the Isomorphism Problem for $\mathcal B_K (G)$, and show that, apart from a specific factorial exception, the monoid structure determines the structure of the torsion quotient group, up to isomorphism, as well as the cardinality of the normal subgroup (see Theorem~\ref{thm:iso}).

% % % % % % % % % % % % % % % % % % % % % % % % % % % % % % % % % % % % % % % % % % % % % % % % % % % % % %
% % %             % % % % % % %               % % % % % % %               % % % % % % %               % % %
% % %             % % % % % % %               % % % % % % %               % % % % % % %               % % %
% % % % % % % % % % % % % % % % % % % % % % % % % % % % % % % % % % % % % % % % % % % % % % % % % % % % % %

\medskip
\section{Preliminaries} \label{2}
\medskip

We denote by $\mathbb{N}$ the set of positive integers and $\mathbb N_0 := \mathbb N \cup \{ 0 \}$.
For real numbers $a,b\in\mathbb{R}$, we write $[a,b]=\{x\in\mathbb{Z} \mid a\le x\le b\}$ for the discrete interval from $a$ to $b$.

By a \textit{monoid}, we always mean a commutative, cancellative semigroup with the identity element, and we usually use the multiplicative notation.
Let $H$ be a monoid.
We denote by
\begin{itemize}
\item $H^{\times}$ the group of invertible elements of $H$,

\smallskip
\item $\mathsf q (H)$ the quotient group of $H$, and

\smallskip
\item $\mathcal A (H)$ the set of atoms (or irreducible elements) of $H$.
\end{itemize}
Then, $H$ is said to be {\it reduced} if $H^{\times}$ is trivial, and we denote by $H_{\text{red}} := \{ aH^{\times} \mid a \in H \}$ the associated reduced semigroup of $H$.
For monoids $H \subseteq D$, we say that $H$ is a {\it divisor-closed submonoid} of $D$ if, for all $a \in H$ and $x \in D$, we have that $x \mid a$ in $D$ implies $x \in H$.
If $a \in H$, then
\[
  [[a]] := \{ x \in H \mid x \mbox{ is a divisor of $a^{n}$ for some } n \in \mathbb N \} \subseteq H
\]
is the smallest divisor-closed submonoid of $H$ containing $a$.
The monoid $H$ is said to be
\begin{itemize}
\item {\it finitely generated} if it has a finite generating set, and

\smallskip
\item {\it locally finitely generated} if $[[a]]_{\text{red}}$ is finitely generated for every $a \in H$.
\end{itemize}
In other words, $H$ is finitely generated if and only if $H_{\text{red}}$ and $H^{\times}$ are finitely generated, and $H$ is locally finitely generated if and only if, for every $a \in H$, there are only finitely many atoms in $H_{\text{red}}$ dividing some power of $a$.
Locally finitely generated monoids have received a lot of attention in recent literature (see, for instance, \cite{Zh19,An22}).

For a set $P$, we let $F = \mathcal{F}(P)$ denote the free abelian monoid with basis $P$. Every $a\in \mathcal{F}(P)$ has a unique representation in the form
\[
  a=\prod_{p\in P}p^{\mathsf v_p(a)} \in \mathcal F (P) \,,
\]
where $\mathsf v_p \colon F \to \mathbb N_0$ is the $p$-adic valuation of $a$.

We also denote by
\begin{itemize}
    \item $H'=\{x\in \mathsf{q}(H) \mid \textup{there is an } N \textup{ such that } x^n\in H \textup{ for all } n\geq N\}$ the \textit{seminormal closure} of $H$,

    \smallskip
   \item  $\widetilde{H}=\{x\in \mathsf{q}(H) \mid x^N\in H \textup{ for some } N\in \N\}$ the \textit{root closure} of $H$, and

   \smallskip
   \item $\widehat{H}=\{x\in \mathsf{q}(H) \mid \textup{there is a } c\in H \textup{ such that } cx^n\in H \textup{ for all } n\in \N\}$ the \textit{complete integral closure} of $H$.
\end{itemize}
Note that $H\subseteq H'\subseteq \widetilde{H}\subseteq \widehat{H}\subseteq \mathsf q(H)$.
The monoid $H$ is called \textit{seminormal} if $H=H'$, \textit{root closed} if $H=\widetilde{H}$, and \textit{completely integrally closed} if $H=\widehat{H}$.

A monoid homomorphism $\varphi \colon H \to D$ is called
\begin{itemize}
    \item \textit{cofinal} if, for every $x\in D$, there exists $h\in H$ such that $x \mid \varphi(h)$,

    \smallskip
    \item a \textit{divisor homomorphism} if, for all $a,b\in H$, the divisibility relation $\varphi(a)\mid \varphi(b)$ implies $a\mid b$,

    \smallskip
    \item a \textit{divisor theory} if $D$ is free abelian, $\varphi$ is a divisor homomorphism, and, for all $x\in D$, there exist $a_1,\ldots, a_n\in H$ such that $x=\gcd (\varphi(a_1),\ldots, \varphi(a_n))$, and

    \smallskip
    \item a \textit{transfer homomorphism} if the following conditions hold:
    		\begin{itemize}
		\smallskip
		\item[\bf{(T1)}] $D=\varphi(H)D^\times$ and $\varphi^{-1}(D^\times)=H^\times$.
		
		\smallskip
		\item[\bf{(T2)}] If $u\in H$ and $b,c\in D$ with $\varphi(u)=b \cdot c$, then there exist $v,w\in H$ such that $u=v \cdot w$, $\varphi(v)D^\times=bD^\times$ and $\varphi(w)D^\times=cD^\times$.
		\end{itemize}
\end{itemize}
Let $\varphi \colon H \to D$ be a monoid homomorphism.
Then, the group
\[
  \mathcal C (\varphi) = \mathsf q (D) / \mathsf q \big( \varphi (H) \big) = \{ [a]_{\varphi} \mid a \in \mathsf q (D) \}
\]
is called the {\it class group} of $\varphi$, where $[a]_{\varphi} := a \mathsf q \big( \varphi (H) \big)$ is the class containing $a \in \mathsf q (D)$.
We usually use the additive notation for $\mathcal C (\varphi)$ and $[1]_{\varphi}$ is the identity element in $\mathcal C (\varphi)$.

Suppose that $H \subseteq D$.
Then, $\varphi \colon H \hookrightarrow D$ is injective, and by identifying $H$ and $\varphi(H)$, we denote by $\mathcal C (\varphi) = \mathsf q (D) / \mathsf q (H) = \{ a \mathsf q (H) \mid a \in \mathsf q (D) \}$.
Hence,
\[
  D / H := \{ a \mathsf q (H) \mid a \in D \} \subseteq \mathcal C (\varphi)
\]
is a submonoid with quotient group $\mathsf q (D/H) = \mathcal C (\varphi)$.
It is well-known that $D/H$ is a group if and only if the injection $\varphi$ is cofinal.
In particular, if $D/H$ is finite, or $\mathsf q (D) / \mathsf q (H)$ is a torsion group, then $D/H = \mathsf q (D) / \mathsf q (H)$ (see \cite[Corollary 2.4.3]{Ge-HK06} for details).

\subsection{Arithmetic concepts for monoids}
Let $H$ be an atomic monoid.
The free abelian monoid $\mathsf Z (H) := \mathcal F \big( \mathcal A (H_{\text{red}}) \big)$ is called the {\it factorization monoid} of $H$, and the unique homomorphism $\pi_H \colon \mathsf Z (H) \to H_{\text{red}}$ given by $\pi_H (u) = u$ for $u \in \mathcal A \big( H_{\text{red}} \big)$ is called the {\it factorization homomorphism} of $H$.
For every $a \in H$, we denote by
\begin{itemize}
\item $\mathsf Z (a) := \pi^{-1}_H \big( aH^{\times} \big) \subseteq \mathsf Z (H)$ the {\it set of factorizations} of $a$, and by

\smallskip
\item $\mathsf L (a) := \{ |z| \mid z \in \mathsf Z (a) \} \subseteq \mathbb N_0$ the {\it set of lengths} of $a$.
\end{itemize}
As usual, $a \in H^{\times}$ if and only if $\mathsf L (a) = \{ 0 \}$, and $a \in \mathcal A (H)$ if and only if $\mathsf L (a) = \{ 1 \}$.
The monoid $H$ is said to be
\begin{itemize}
\item {\it atomic} if $\mathsf Z (a) \neq \emptyset$ for every $a \in H$,

\smallskip
\item {\it factorial} if $\left| \mathsf Z (a) \right| = 1$ for every $a \in H$,

\smallskip
\item {\it half-factorial} if $| \mathsf L (a) | = 1$ for every $a \in H$,

\smallskip
\item an {\it FF-monoid} (a finite factorization monoid) if $\mathsf Z (a)$ is finite and non-empty for every $a \in H$, and

\smallskip
\item a {\it BF-monoid} (a bounded factorization monoid) if $\mathsf L (a)$ is finite and non-empty for every $a \in H$.
\end{itemize}
A monoid $H$ is half-factorial (resp., an FF-monoid, a BF-monoid) if and only if $H_{\text{red}}$ has the respective property.
Clearly, every factorial monoid is both a half-factorial monoid and an FF-monoid.
Moreover, every half-factorial monoid and every FF-monoid is a BF-monoid.
However, a half-factorial monoid need not be an FF-monoid (see \cite[Example 1.6.10]{Ge-HK06}).

We denote by $\mathcal L (H) = \{ \mathsf L (a) \mid a \in H \}$ the {\it system of sets of lengths} of $H$.
If $k \in \mathbb N$ and $H \neq H^{\times}$, then
\[
  \mathcal U_k (H) = \bigcup_{k\in \mathsf L, \mathsf L\in \mathcal{L}(H)} \mathsf L \subseteq \mathbb N
\]
denotes the {\it union of sets of lengths containing k}, and we set $\rho_k (H) = \sup \mathcal U_k (H)$ the {\it k-th elasticity} of $H$.
Moreover,
\begin{equation} \label{eq:rho}~
  \rho (H) = \sup \left\{ \frac{\sup \mathsf L (a)}{\min \mathsf L (a)} \, \Big\vert \, a \in H \right\} =  \sup \left\{ \frac{\rho_k (H)}{k} \, \Big\vert \, k \in \mathbb N \right\} = \lim_{k \to \infty} \frac{\rho_k (H)}{k}
\end{equation}
is the {\it elasticity} of $H$ (see \cite[Proposition 1.1.13]{Ge-Gr-Zh26} for details).
We say that $H$ has $\textit{accepted elasticity}$ if there exists some \(a\in H\) such that $\rho(H) = \rho \big( \mathsf L(a) \big)$.
With these notions, the monoid $H$ is half-factorial if and only if $\mathcal U_k (H) = \{ k \}$ for every $k \in \mathbb N$ if and only if $\rho_k (H) = k$ for every $k \in \mathbb N$ if and only if $\rho (H) = 1$.

\subsection{Krull and Transfer Krull monoids}
All arithmetic concepts discussed before have been well understood for Krull monoids.
Let $H$ be a monoid.
We denote by
\begin{itemize}
\item $s$-$\spec (H)$ the set of prime $s$-ideal of $H$,

\smallskip
\item $\mathfrak X (H)$ the set of minimal non-empty prime $s$-ideals of $H$, and

\smallskip
\item $H_{\mathfrak p} = (H \setminus \mathfrak p)^{-1} H \subseteq \mathsf q (H)$ the localization at a prime ideal $\mathfrak p \in s$-$\spec (H)$.
\end{itemize}
The monoid $H$ is said to be {\it weakly Krull} if the following two conditions hold:
\begin{itemize}
\item[(i)] $H = \bigcap_{\mathfrak p \in \mathfrak X (H)} H_{\mathfrak p}$.

\smallskip
\item[(ii)] Each element of $H$ is contained in only finitely many prime $s$-ideals in $\mathfrak X (H)$.
\end{itemize}
A weakly Krull domain generalizes one-dimensional Noetherian domains, but they need not be integrally closed.
For instance, every order in algebraic number fields, and every Cohen-Kaplansky domain, is weakly Krull.

The monoid $H$ is a {\it Krull monoid} if one of the following equivalent conditions is satisfied (see \cite[Chapter 2]{Ge-HK06}):
\begin{itemize}
\item[(a)] $H$ is a weakly Krull monoid and $H_{\mathfrak p}$ is a discrete valuation monoid for all $\mathfrak p \in \mathfrak X (H)$, i.e., $H_{\mathfrak p}$ is factorial and has exactly one prime element.

\smallskip
\item[(b)] $H$ is completely integrally closed and a Mori monoid, i.e., $H$ satisfies the ascending chain condition on divisorial ideals.

\smallskip
\item[(c)] $H$ has a divisor theory.

\smallskip
\item[(d)] There is a divisor homomorphism from $H$ into a free abelian monoid.
\end{itemize}
If $H$ is a Krull monoid, then a divisor theory $\varphi \colon H \to F$ (by condition (c)) is unique up to isomorphism (see \cite[Theorem 2.4.7]{Ge-HK06}), and
\[
  \mathcal C (H) := \mathcal C (\varphi) = \mathsf q (F) / \mathsf q \big( \varphi (H) \big)
\]
is called the {\it class group} of $H$.

The monoid $H$ is said to be a \textit{transfer Krull monoid} if it has a transfer homomorphism to a Krull monoid $M$.
The most significant feature of a transfer homomorphism from $H$ to $M$ is as follows: it allows us to pull back arithmetic properties from $M$, which we hope is simpler, to the original monoid $H$.

\smallskip
\begin{lemma} $($\cite[Lemma 1.4.2]{Ge-Gr-Zh26}$)$ \label{lem:trans}~
Let $H$ and $M$ be atomic monoids and $\theta \colon H \to M$ be a transfer homomorphism.
\begin{enumerate}
\item $u \in \mathcal A (H)$ if and only if $\theta (u) \in \mathcal A (M)$.

\smallskip
\item $\mathcal L (H) = \mathcal L (M)$, so that $\mathcal U_k (H) = \mathcal U_k (M)$, $\rho_k (H) = \rho_k (M)$ for every $k \in \mathbb N$, and $\rho (H) = \rho (M)$.

\smallskip
\item $H$ is half-factorial (or has accepted elasticity) if and only if $M$ is half-factorial (or has accepted elasticity).
\end{enumerate}
\end{lemma}

Clearly, every Krull monoid is a transfer Krull monoid, and we refer the reader to \cite{Ge16,Ge-Zh20,Ba-Re22,Ba-Po26,Ge-Gr-Zh26} for background and further examples of transfer Krull monoids.

\subsection{Class semigroups and $C$-monoids}
Let $H \subseteq F$ be monoids.
The \textit{class semigroup} of $H$ in $F$ is the quotient semigroup
\[
  \mathcal C (H,F) = \big\{ [y]^{F}_H \mid y \in F \big\} \,,
\]
where, for $y, y' \in F$, we set $y \sim_H y'$ if and only if $y^{-1}H\cap F=(y')^{-1}H\cap F$.
The congruence class of $y \in F$ is denoted by $[y]^{F}_H$.
As usual, the operation is defined by $[x]^{F}_H + [y]^{F}_H = [xy]^{F}_H$, which makes $\mathcal C(H,F)$ into a semigroup.
The \textit{reduced class semigroup} of $H$ in $F$ is the subsemigroup
\[
  \mathcal C^*(H,F) = \big\{[y]^{F}_H \mid y \in (F \setminus F^\times) \cup \{ 1 \} \big\} \,.
\]

A monoid $H$ is called a {\it C-monoid} if $H$ is a submonoid of a factorial monoid $F$ such that $H \cap F^{\times} = H^{\times}$ and the reduced class semigroup $\mathcal C^*(H,F)$ is finite.
A C-monoid is a multiplicative model for studying the arithmetic of non-integrally closed Noetherian domains, and thus it serves as a non-completely integrally closed analogue of Krull monoids.
We refer the reader to \cite{Re13,Ge-Zh19,Oh22} for the algebraic structures of C-monoids and their class semigroup.

\subsection{Product-$K$ sequences} Let $G$ be a group with identity $1_G$, $G_0 \subseteq G$ be a subset, and $\mathcal F(G_0)$ be the multiplicatively written free abelian monoid with basis $G_0$.
In combinatorial language, the elements of $\mathcal F(G_0)$ are called
\textit{sequences} over $G$.
A sequence $S \in \mathcal F(G_0)$ will be written in the form
\begin{equation} \label{eq:sequence}~
  S = g_1 \bdot \ldots \bdot g_{\ell} = {\small \prod}^{\bullet}_{g \in G_0} g^{[\mathsf v_g (S)]} \,,
\end{equation}
where $g_1, \ldots, g_{\ell} \in G_0$, $\mathsf v_g (S) = |\{ i \in [1,\ell] \mid g_i = g \}|$ is the {\it multiplicity} of $g$ in $S$, and $|S| = \ell = \sum_{g \in G_0} \mathsf v_g (S)$ is the {\it length} of $S$.
We denote by $\supp (S) = \{ g \in G_0 \mid \mathsf v_g (S) \ge 1 \}$ the {\it support} of $S$.
For sequences $S, T \in \mathcal F (G_0)$, we say that $T$ is a {\it divisor} of $S$ in $\mathcal F (G_0)$, denoted by $T \mid S$, if $S = T \bdot W$ for some $W \in \mathcal F (G_0)$.
In this case, $T$ is called a {\it subsequence} of $S$, and $S \bdot T^{[-1]} := W = {\small \prod}^{\bullet}_{g \in G_0} g^{[\mathsf v_g (S) - \mathsf v_g (T)]} \in \mathcal F (G_0)$ denotes the subsequence of $S$ obtained by deleting the terms from $T$ in $S$.
The {\it set of products} of $S$ is denoted by
\[
  \pi (S) = \{ g_{\sigma (1)} \cdots g_{\sigma (\ell)} \in G \mid \sigma \mbox{ is a permutation on} [1, \ell] \} \,.
\]
A generic element of $\pi (S)$ will be denoted by $\sigma (S)$, and in particular, if $G$ is abelian, then $\pi (S) = \{ \sigma (S) \}$.
For a normal subgroup $K$ of $G$, denoted by $K \unlhd G$, the sequence $S$ is called {\it product-$K$} if $\pi (S) \cap K \neq \emptyset$, and then the set $\mathcal B_K (G_0)$ of all product-$K$ sequences forms a submonoid of $\mathcal F (G_0)$ with $\mathcal A_K (G_0) := \mathcal A \big( \mathcal B_K (G_0) \big)$.
Moreover, we define
\[
  \mathsf D_K (G_0) := \mathsf D \big( \mathcal B_K (G_0) \big) = \sup \{ |S| \mid S \in \mathcal A_K (G_0) \} \,.
\]
Since the commutator subgroup $G'$ is normal in $G$, we observe that $KG'$ is also a normal subgroup of $G$ containing $G'$, so that $G/ KG'$ is an abelian group.
Thus, for any $S \in \mathcal F (G_0)$, we can view $\pi (S)$ as a subset of a $KG'$-coset.
We denote by $\phi_K \colon G \to G/K$ the natural group epimorphism.
By abuse of notation, we also denote $\phi_K \colon \mathcal F (G) \to \mathcal F \left( G/K \right)$ for the extension of $\phi_K$ to sequences, i.e., $\phi_K (g_1 \bdot \ldots \bdot g_{\ell}) = \phi_K (g_1) \bdot \ldots \bdot \phi_K (g_{\ell})$.
With this notation, we denote by $\ord_K (g) := \ord \big( \phi_K (g) \big)$ for every $g \in G$, and moreover it is straightforward to obtain that, for every $S \in \mathcal F (G)$,
\begin{equation} \label{eq:phi}~
  \pi (S) \cap K \neq \emptyset \quad \mbox{ if and only if } \quad \phi_K (S) \in \mathcal B \left( G/K \right) \,.
\end{equation}

The case where $K = \{ 1_G \}$ is trivial is of particular interest.
In this case, product-$K$ sequences are called product-one sequences, and $\mathcal B (G_0):= \mathcal B_K (G_0)$ denotes the monoid of product-one sequences.
The following results present the algebraic structure of the monoid $\mathcal B (G)$.

\smallskip
\begin{proposition} $($\cite[Theorem 3.14]{Fa-Zh23} and \cite[Corollary 3.12]{Oh19}$)$ \label{pro:B(G)}~
Let $G$ be a group.
\begin{enumerate}
\item The following statements are equivalent:
      \begin{enumerate}
      \smallskip
      \item[(a)] $G$ is abelian.

      \smallskip
      \item[(b)] $\mathcal B (G)$ is a Krull monoid.

      \smallskip
      \item[(c)] $\mathcal B (G)$ is a transfer Krull monoid.

      \smallskip
      \item[(d)] The embedding $\mathcal B (G) \hookrightarrow \mathcal F (G)$ is a divisor homomorphism.
      \end{enumerate}

\smallskip
\item $\mathcal B (G)$ is a C-monoid if and only if $G$ is finite.

\smallskip
\item $\mathcal B (G)$ is seminormal if and only if $|G'| \le 2$. In particular, if $G$ is finite, then $\mathcal B (G)$ is a seminormal C-monoid if and only if the class semigroup of $\mathcal B (G)$ in $\mathcal F (G)$ is Clifford if and only if $|G'| \le 2$.
\end{enumerate}
\end{proposition}

\smallskip
If $G$ is abelian, then the arithmetic of the monoid $\mathcal B (G)$ reflects that of general Krull monoids by the following proposition together with Lemma~\ref{lem:trans}.

\smallskip
\begin{proposition} $($\cite[Theorem 1.4.5]{Ge-Gr-Zh26}$)$ \label{pro:Krull}~
Let $H$ be a Krull monoid with divisor theory $\varphi \colon H \to \mathcal F (P)$ and the class group $G$ in which each class contains a prime divisor.
Let $\widetilde{\boldsymbol{\beta}} \colon \mathcal F (P) \to \mathcal F (G)$ denote the unique homomorphism defined by $\widetilde{\boldsymbol{\beta}} (p) = [p] \in G$ for all $p \in P$.
Then, the homomorphism $\boldsymbol{\beta} \colon H \to \mathcal B (G)$ is a transfer homomorphism.
In particular, $\mathsf L (a) = \mathsf L \big( \boldsymbol{\beta} (a) \big)$ for all $a \in H$ and $\mathcal L (H) = \mathcal L \big( \mathcal B (G) \big)$.
\end{proposition}

If $G$ is torsion, then $\mathcal B (G)$ is a BF-monoid and since $G/G'$ is abelian, we obtain that $\widehat{\mathcal B (G)} = \mathcal B_{G'} (G)$ (see \cite[Lemma 3.4]{Fa-Zh23}).
In particular, if $G$ is finite, then $\mathcal B (G)$ is a finitely generated C-monoid.
Moreover, $\mathsf D (G):= \mathsf D_{\{ 1_G\}} (G)$ is the usual Davenport constant of $G$, a classical combinatorial invariant whose study is a central topic in arithmetic combinatorics.
Consequently, because of its connections with various areas of mathematics, the study of combinatorial invariants of the monoid of product-one sequences has attracted wide attention in both invariant theory and factorization theory (see \cite{Oh-Zh20b,Zh21,Qu-Li22,Qu-Li-Tee22,Fa-Zh23,Av-Br-Ri23,Ri25,Qu-Li-Tee25,Oh-Ri-Zha-Zh26} for recent progress).

As another extremal case, $\mathcal B_K (G) = \mathcal F (G)$ if $K = G$, and so it is clear that
\begin{equation} \label{eq:sub}~
  \mathcal B (G) = \mathcal B_{\{ 1_G\}} (G) \subseteq \mathcal B_K (G) \subseteq \mathcal B_G (G) = \mathcal F (G) \,.
\end{equation}

Let $S \in \mathcal B_K (G_0)$. The next simple lemma shows that any ordered product in $\pi (S) \cap K$  can have its order cyclically shifted.

\smallskip
\begin{lemma} \label{lem:shift}~
If $S = g_1 \bdot \ldots \bdot g_{\ell} \in \mathcal B_K (G_0)$ with $g_1 \cdots g_{\ell} \in K$, then $g_i \cdots g_{\ell} g_1 \cdots g_{i-1} \in K$ for every $i \in [1,\ell]$.
\end{lemma}

\begin{proof}
Let $S = g_1 \bdot \ldots \bdot g_{\ell} \in \mathcal B_K (G_0)$ with $g_1 \cdots g_{\ell} = k$ for some $k \in K$.
For $i \in [1,\ell]$, by multiplying on the left by $g^{-1}_{i-1} \cdots g^{-1}_1$ and on the right by $g_1 \cdots g_{i-1}$, the normality of $K$ in $G$ ensures that
\[
  g_i \cdots g_{\ell} g_1 \cdots g_{i-1} = g^{-1}_{i-1} \cdots g^{-1}_1 k g_1 \cdots g_{i-1} = (g_1 \cdots g_{i-1})^{-1} k (g_1 \cdots g_{i-1}) \in K \,.  \qedhere
\]
\end{proof}

% % % % % % % % % % % % % % % % % % % % % % % % % % % % % % % % % % % % % % % % % % % % % % % % % % % % % %
% % %             % % % % % % %               % % % % % % %               % % % % % % %               % % %
% % %             % % % % % % %               % % % % % % %               % % % % % % %               % % %
% % % % % % % % % % % % % % % % % % % % % % % % % % % % % % % % % % % % % % % % % % % % % % % % % % % % % %

\medskip
\section{Main results} \label{3}
\medskip

\smallskip
We start with the characterization of divisor-closed submonoids of the monoid $\mathcal B_K (G_0)$, which is well known when $K$ is trivial.

\smallskip
\begin{lemma} \label{lem:divisor}~
Let $G$ be a group, $K \unlhd G$, and $G_0 \subseteq G$ be a subset.
\begin{enumerate}
\item The monoid $\mathcal B_K (G_0)$ is an FF-monoid.

\smallskip
\item A submonoid $H \subseteq \mathcal B_K (G_0)$ is divisor-closed if and only if $H = \mathcal B_K (G_1)$ for some $G_1 \subseteq G_0$.
\end{enumerate}
\end{lemma}

\begin{proof}
1. As a factorial monoid, $\mathcal F (G_0)$ is a reduced FF-monoid.
Since $\mathcal B_K (G_0)$ is also a reduced monoid, it follows by \cite[Theorem 1.5.6]{Ge-HK06} that $\mathcal B_K (G_0)$ is an FF-monoid.

\smallskip
2. For any $G_1 \subseteq G_0$, $\mathcal B_K (G_1) \subseteq \mathcal B_K (G_0)$ is clearly a divisor-closed submonoid.
For the converse, let $H \subseteq \mathcal B_K (G_0)$ be a divisor-closed submonoid.
We set $G_1 := \bigcup_{T \in H} \supp (T) \subseteq G_0$, and then $H \subseteq \mathcal B_K (G_1)$.
Now let $S = g_1 \bdot \ldots \bdot g_{\ell} \in \mathcal B_K (G_1)$.
For each $i \in [1,\ell]$, we can take $T_i \in H$ such that $g_i \in \supp (T_i)$, and by Lemma~\ref{lem:shift}, we infer that $\pi \big( T_i \bdot g^{[-1]}_i \big) \cap g^{-1}_i K \neq \emptyset$.
If $\sigma$ is a permutation on $[1,\ell]$ such that $g_{\sigma (1)} \cdots g_{\sigma (\ell)} \in K$, then $g^{-1}_{\sigma (\ell)} \cdots g^{-1}_{\sigma (1)} \in K$, and since $K$ is normal in $G$, we have that
\[
  \emptyset \neq \pi \left( T_{\sigma (\ell)} \bdot g^{[-1]}_{\sigma (\ell)} \right) \cdots \pi \left( T_{\sigma (1)} \bdot g^{[-1]}_{\sigma (1)} \right) \cap \big( g^{-1}_{\sigma (\ell)} K \big) \cdots \big( g^{-1}_{\sigma (1)} K \big) \subseteq \pi \left( T_1 \bdot \ldots \bdot T_{\ell} \bdot S^{[-1]} \right) \cap K \,.
\]
Therefore, $S \mid T_1 \bdot \ldots \bdot T_{\ell}$ in $\mathcal B_K (G_0)$, and since $H$ is divisor-closed in $\mathcal B_K (G_0)$, we have that $S \in H$.
\end{proof}

\smallskip
\begin{proposition} \label{pro:local}~
Let $G$ be a group and $K \unlhd G$.
\begin{enumerate}
\item s-$\spec \big( \mathcal B_K (G) \big) = \{ \mathfrak p_{\Gamma} \mid \Gamma \subseteq G \}$, where
	\[
	  \mathfrak p_{\Gamma} = \{ T \in \mathcal B_K (G) \mid \mathsf v_g (T) \ge 1 \, \mbox{ for some } \, g \in \Gamma \} \,\, \mbox{ for } \,\, \Gamma \subseteq G \,.
	\]
	In particular, $\mathfrak X \big( \mathcal B_K (G) \big) = \{ \mathfrak p_g \mid g \in G \}$.

\smallskip
\item If $G/K$ is a torsion group, then
	\[
	  \widetilde{\mathcal B_K (G)} = \widehat{\mathcal B_K (G)} = \{ S \in \mathcal F (G) \mid \pi (S) \subseteq KG' \} = \bigcap_{\mathfrak p \in \mathfrak X (\mathcal B_K (G))} \mathcal B_K (G)_{\mathfrak p} \,.
	\]
\end{enumerate}
\end{proposition}

\begin{proof}
1. For any subset $\Gamma \subseteq G$, $\mathfrak p_{\Gamma}$ is clearly a prime $s$-ideal of $\mathcal B_K (G)$.
For the reverse inclusion, we pick $\mathfrak P \in s$-$\spec \big( \mathcal B_K (G) \big)$.
Then, $\mathcal B_K (G) \setminus \mathfrak P$ is a divisor-closed submonoid of $\mathcal B_K (G)$.
By Lemma~\ref{lem:divisor}, $\mathcal B_K (G) \setminus \mathfrak P = \mathcal B_K (G_0)$ for some $G_0 \subseteq G$, which means that the ideal $\mathfrak P$ consists of all $T \in \mathcal B_K (G)$ such that $\mathsf v_g (T) \ge 1$ for some $g \in G \setminus G_0$, whence $\mathfrak P = \mathfrak p_{(G \setminus G_0)}$.

For the ``in particular'' statement, we observe that $\Gamma_1 \subseteq \Gamma_2$ implies that $\mathfrak p_{\Gamma_1} \subseteq \mathfrak p_{\Gamma_2}$, and thus $\mathfrak X \big( \mathcal B_K (G) \big) \subseteq \{ \mathfrak p_g \mid g \in G \}$.
For the reverse inclusion, it suffices to show that $\mathfrak p_g$ and $\mathfrak p_h$ are not comparable for any distinct $g, h \in G$.
Let $g, h \in G$ with $g \neq h$.
If $\ord_K (g) < \infty$, then we clearly have that $g^{[\ord_K (g)]} \in \mathfrak p_g \setminus \mathfrak p_h$.
Suppose that $\ord_K (g) = \infty$.
Then, $g$ must have infinite order in $G$, which ensures that $g, g^{-1}, g^{-2}$ are pairwise distinct.
Hence, $T_1 := g \bdot g^{-1}$ and $T_2 := g^{[2]} \bdot g^{-2}$ are both product-$K$ sequences with $\supp (T_1) \cap \supp (T_2) = \{ g \}$.
Since $h \neq g$, it follows that either $h \notin \supp (T_1)$ or $h \notin \supp (T_2)$.
In either case, one of them belongs to $\mathfrak p_g$ but not to $\mathfrak p_h$.
Therefore, $\mathfrak p_g \nsubseteq \mathfrak p_h$ for any distinct $g, h \in G$.

\smallskip
2. We show the following inclusions:
\[
  \widehat{\mathcal B_K (G)} \, \overset{(a)}{\subseteq} \, \{ S \in \mathcal F (G) \mid \pi (S) \subseteq KG' \}  \, \overset{(b)}{\subseteq} \,  \small{\bigcap}_{\mathfrak p \in \mathfrak X (\mathcal B_K (G))} \mathcal B_K (G)_{\mathfrak p} \, \overset{(c)}{\subseteq} \, \widetilde{\mathcal B_K (G)} \,.
\]

For (a), since $\mathcal B_K (G) \subseteq \mathcal F (G)$ and $\mathcal F (G)$ is factorial, it follows that $\widehat{\mathcal B_K (G)} \subseteq \mathcal F (G)$.
Thus, for any $\frac{S_1}{S_2} \in \widehat{\mathcal B_K (G)}$ with $S_1, S_2 \in \mathcal B_K (G)$, we can write $S_1 = S_2 \bdot T$ for some $T \in \mathcal F (G)$.
Since $S_1, S_2 \in \mathcal B_K (G)$, there exist $k_1, k_2 \in K$ such that $k_1 \in \pi (S_1) \cap K$ and $k_2 \in \pi (S_2) \cap K$.
Since $G/G'$ is abelian, it follows that $\pi (S_1) \subseteq k_1 G'$ and $\pi (S_2) \subseteq k_2 G'$.
Then, $S_1 = S_2 \bdot T$ ensures that $\pi (S_2) \pi (T) \subseteq \pi (S_1) \subseteq k_1 G'$, so that $\pi (T) \subseteq (k^{-1}_2 k_1)G' \subseteq KG'$.
Thus, $\widehat{\mathcal B_K (G)} \subseteq \{ S \in \mathcal F (G) \mid \pi (S) \subseteq KG' \}$.

For (b), it suffices to show that $kx \in \bigcap_{\mathfrak p \in \mathfrak X (\mathcal B_K (G))} \mathcal B_K (G)_{\mathfrak p}$ for all $k \in K$ and $x \in G'$.
Indeed, if this holds, then for any $S \in \mathcal F (G)$ with $\alpha \in \pi (S) \subseteq KG'$ and for $g \in G$, there exists $T \in \mathcal B_K (G) \setminus \mathfrak p_g$ such that $\alpha \bdot T \in \mathcal B_K (G)$.
This ensures that $S \bdot T \in \mathcal B_K (G)$, so that $S = \frac{S \bdot T}{T} \in \mathcal B_K (G)_{\mathfrak p_g}$.

We first consider the case where $x$ is a commutator element in $G$.
Let $k \in K$ and $x = ghg^{-1}h^{-1}$ with $g, h \in G$.
If $x = 1_G$, then $kx \in \mathcal B_K (G)_{\mathfrak p}$ for every $\mathfrak p \in \mathfrak X \big( \mathcal B_K (G) \big)$.
Suppose that $x \neq 1_G$.
Since $G/K$ is a torsion group, we obtain that
\[
  (kx) \bdot h^{[\ord_K (h)]} \bdot g \bdot g^{-1} \,\, \mbox{ and } \,\, (kx) \bdot (hg)^{[\ord_K (hg)]} \bdot g \bdot g^{-1} \,\, \in \,\, \mathcal B_K (G) \,.
\]
Indeed, since $h^{\ord_K (h)} \in K$, it follows that $h^{\ord_K (h) - 1} = k_1h^{-1}$ for some $k_1 \in K$, and so the normality of $K$ ensures that $(h^{\ord_K (h)-1}) g^{-1} (kx) h g = (k_1h^{-1})g^{-1}k(gh) = (gh)^{-1}(k_2 k)(gh) \in K$ for some $k_2 \in K$.
Also, it follows by the same argument that $g \big( (hg)^{\ord_K (hg)-1} \big)g^{-1} (kx) (hg) \in K$.
Moreover, since $G/K$ is a torsion group, we also obtain that
\[
  (kx) \bdot g^{[\ord_K (g)]} \bdot h \bdot h^{-1} \,\, \mbox{ and } \,\, (kx) \bdot (hg)^{[\ord_K (hg)]} \bdot h \bdot h^{-1} \,\, \in \,\, \mathcal B_K (G) \,.
\]
Indeed, since $(hg)^{\ord_K (hg)} \in K$, it follows that $(hg)^{\ord_K (hg) - 1} = k_1(hg)^{-1}$ for some $k_1 \in K$, and thus by the normality of $K$, we infer that $h^{-1} \big( (hg)^{\ord_K (hg)-1} \big) h (kx) (hg) = h^{-1} \big(k_1 (hg)^{-1} \big) h k (gh) = (gh)^{-1}(k_2 k) (gh) \in K$ for some $k_2$.
Also, it follows by the same argument that $h^{-1} (g^{\ord_K (g) - 1}) (kx) h g \in K$.
Therefore, we obtain that
\[
  kx = \frac{(kx) \bdot T_1}{T_1} = \frac{(kx) \bdot T_2}{T_2} \, \in \, \mathsf q \big( \mathcal B_K (G) \big) \,,
\]
where $T_1 \in \{ h^{[\ord_K (h)]} \bdot g \bdot g^{-1}, \, (hg)^{[\ord_K (hg)]} \bdot g \bdot g^{-1} \}$ and $T_2 \in \{ g^{[\ord_K (g)]} \bdot h \bdot h^{-1}, \, (hg)^{[\ord_K (hg)]} \bdot h \bdot h^{-1} \}$.
Since $x \neq 1_G$, it follows that $\{h, g, g^{-1} \} \cap \{ (hg), g, g^{-1} \} = \{ g, g^{-1} \}$ and $\{ g, h, h^{-1} \} \cap \{ (hg), h, h^{-1} \} = \{ h, h^{-1} \}$, whence
\[
  A := \{h, g, g^{-1} \} \cap \{ (hg), g, g^{-1} \}  \cap \{ g, h, h^{-1} \} \cap \{ (hg), h, h^{-1} \} = \emptyset \,.
\]
Thus, $kx \in \mathcal B_K (G)_{\mathfrak p_z}$ for all $z \in G \setminus A = G$, equivalently by item 1, $kx \in \mathcal B_K (G)_{\mathfrak p}$ for all $\mathfrak p \in \mathfrak X \big( \mathcal B_K (G) \big)$.

Now let $k \in K$ and $x \in G'$.
Then, $x = \prod_{i \in [1,\ell]} g_i h_i g^{-1}_i h^{-1}_i$ for some $g_1, \ldots, g_{\ell}, h_1, \ldots, h_{\ell} \in G$, and for each $i \in [1,\ell]$, we set $x_i = g_i h_i g^{-1}_i h^{-1}_i$.
Let $g \in G$.
Then, as we have shown in the previous paragraph, there exist $T_1, \ldots, T_{\ell} \in \mathcal B_K (G) \setminus \mathfrak p_g$ such that $(k x_1) \bdot T_1 \in \mathcal B_K (G)$ and $x_i \bdot T_i \in \mathcal B_K (G)$ for $i \in [2,\ell]$.
By Lemma~\ref{lem:shift}, we can choose $\sigma (T_1) \in \pi (T_1)$ such that $k_1 := (kx_1)\sigma(T_1) \in K$, and for each $i \in [2,\ell]$, choose $\sigma (T_i) \in \pi (T_i)$ such that $k_i := x_i \sigma (T_i) \in K$.
Hence, we obtain that
\[
  \begin{aligned}
    (k x_1 \cdots x_{\ell}) \sigma (T_{\ell}) \cdots \sigma (T_1) & = (kx_1 \cdots x_{\ell-1}) \big( x_{\ell} \sigma (T_{\ell}) \big) \big( \sigma (T_{\ell-1}) \cdots \sigma (T_1) \big) \\
                                                                  & = (kx_1 \cdots x_{\ell-1}) k_{\ell} \big( \sigma (T_{\ell-1}) \cdots \sigma (T_1) \big) \,,
  \end{aligned}
\]
and since $K$ is normal in $G$, there exists $k'_{\ell} \in K$ such that
\[
  (kx_1 \cdots x_{\ell-1}) k_{\ell} \big( \sigma (T_{\ell-1}) \cdots \sigma (T_1) \big) = k'_{\ell} (kx_1 \cdots x_{\ell-1}) \big( \sigma (T_{\ell-1}) \cdots \sigma (T_1) \big) \,.
\]
Repeating this argument, we arrive at $(k x_1 \cdots x_{\ell}) \sigma (T_{\ell}) \cdots \sigma (T_1) = k'_1 k'_2 \cdots k'_{\ell}$ for some $k'_1, \ldots, k'_{\ell} \in K$, and thus $\pi \big( (kx) \bdot T_{\ell} \bdot \ldots \bdot T_1 \big) \cap K \neq \emptyset$.
By letting $T := \prod^{\bullet}_{i \in [1,\ell]} T_i$, we infer that $T \in \mathcal B_K (G) \setminus \mathfrak p_g$, and $(kx) \bdot T \in \mathcal B_K (G)$, and hence $kx = \frac{(kx) \bdot T}{T} \in \mathcal B_K (G)_{\mathfrak p_g}$.

For (c), let $\frac{S_1}{S_2} \in \bigcap_{\mathfrak p \in \mathfrak X (\mathcal B_K (G))} \mathcal B_K (G)_{\mathfrak p}$.
If $\supp (S_1) \cap \supp (S_2) \neq \emptyset$, then by deleting common terms, we may assume that $S_1, S_2 \in \mathcal F (G)$ with $\supp (S_1) \cap \supp (S_2) = \emptyset$.
Assume to the contrary that $S_2$ is non-trivial.
Then, there exists at least one $g \in G$ such that $g \in \supp (S_2)$.
Since $\frac{S_1}{S_2} \in \mathcal B_K (G)_{\mathfrak p}$ for all $\mathfrak p \in \mathfrak X \big( \mathcal B_K (G) \big)$, it follows by Item 1 that there exist $T_g \in \mathcal B_K (G)$ and $T'_g \in \mathcal B_K (G) \setminus \mathfrak p_g$ such that $\frac{S_1}{S_2} = \frac{T_g}{T'_g}$.
Since $\supp (S_1) \cap \supp (S_2) = \emptyset$ and $g \in \supp (S_2)$, $S_1 \bdot T'_g = T_g \bdot S_2$ ensures that $g \in \supp (T'_g)$, which means that $T'_g \in \mathfrak p_g$, a contradiction.
Hence, $S_2$ must be a trivial sequence, ensuring that $S_1 = \frac{S_1}{S_2} \in \mathcal F (G) \cap \mathsf q \big( \mathcal B_K (G) \big)$.
Now let $S_1 = g_1 \bdot \ldots \bdot g_{\ell}$.
Since $G/K$ is a torsion group, we may define $n := \lcm \{ \ord_K (g_i) \mid i \in [1,\ell] \}$.
Then, $S^{[n]}_1 \in \mathcal B_K (G)$, and thus we infer that $\frac{S_1}{S_2} = S_1 \in \widetilde{\mathcal B_K (G)}$.
\end{proof}

\smallskip
The following theorem is our main result on algebraic properties of $\mathcal B_K (G)$.
If $G$ is abelian, then Item 3.(a) holds trivially, and all the remaining statements in the theorem coincide with those in the classical abelian setting (see \cite{HK92}).
If $G$ is not necessarily abelian but $K$ is trivial, then Item 3.(a) is equivalent to the assumption that $G$ is abelian (cf. Proposition~\ref{pro:B(G)}).
Consequently, all statements in the theorem reduce to their counterparts in the non-abelian setting developed in recent years (see \cite{Cz-Do-Ge16,Oh19,Oh20,Ge-Gr-Oh-Zh22,Fa-Zh23}).

\smallskip
\begin{theorem} \label{thm:struc}~
Let $G$ be a group and $K \unlhd G$ be such that $G/K$ is a torsion group.
\begin{enumerate}
\item The monoid $\mathcal B_K (G)$ is a BF-monoid, and $\widehat{\mathcal B_K (G)}$ is a Krull monoid.
Moreover, except in the case where $|G| = 2$ and $|K| = 1$, the embedding $\widehat{\mathcal B_K (G)} \hookrightarrow \mathcal F (G)$ is a divisor theory with the class group $\mathcal F (G) / \mathcal B_K (G)$, which is isomorphic to $G/KG'$, and every class contains precisely $|KG'|$ prime divisors.

\smallskip
\item $\big( \mathcal B_K (G) \colon \widehat{\mathcal B_K (G)} \big) = \big\{ S \in \mathcal F (G) \mid \pi \big( \phi_K (S) \big) = \left( G/K \right)' \big\} \neq \emptyset$ if and only if $\left( G/K \right)'$ is finite.

\smallskip
\item The following statements are equivalent:
	\begin{enumerate}
	\smallskip
	\item[(a)] $G' \subseteq K$.
	
	\smallskip
	\item[(b)] $\mathcal B_K (G)$ is a Krull monoid.
	
	\smallskip
	\item[(c)] $\mathcal B_K (G)$ is completely integrally closed.
	
	\smallskip
	\item[(d)] $\mathcal B_K (G)$ is root-closed.
	
	\smallskip
	\item[(e)] $\mathcal B_K (G)$ is a transfer Krull monoid.
	
	\smallskip
	\item[(f)] $\mathcal B_K (G)$ is a weakly Krull monoid.
	
	\smallskip
	\item[(g)] The embedding $\mathcal B_K (G) \hookrightarrow \mathcal F (G)$ is a divisor homomorphism.
	
	\smallskip
	\item[(h)] The class semigroup of $\mathcal B_K (G)$ in $\mathcal F (G)$ is a group.
	\end{enumerate}
	\smallskip
	If this is the case, then the class semigroup of $\mathcal B_K (G)$ in $\mathcal F (G)$ is isomorphic to $G/K$.
\end{enumerate}
\end{theorem}

\begin{proof}
1. It follows by Lemma~\ref{lem:divisor}.1 that $\mathcal B_K (G)$ is an FF-monoid, and thus it is a BF-monoid.
Since $G/K$ is a torsion group, we observe that $g^{[\ord_K (g)]} \in \mathcal B_K (G)$ for every $g \in G$.
Thus, for every $S = g_1 \bdot \ldots \bdot g_{\ell} \in \mathcal F (G)$, if we set $n_S := \lcm \{ \ord_K (g_i) \mid i \in [1,\ell] \}$, then $S^{[n_S]} \in \mathcal B_K (G)$.
This means that the embeddings $\mathcal B_K (G) \hookrightarrow \mathcal F (G)$ and $\widehat{\mathcal B_K (G)} \hookrightarrow \mathcal F (G)$ are both cofinal and $\mathcal F (G) / \mathcal B_K (G)$ is a group itself.
Moreover,
\[
  \mathcal F (G) / \mathcal B_K (G) = \mathsf q \big( \mathcal F (G) \big) / \mathsf q \big( \mathcal B_K (G) \big) = \mathsf q \big( \mathcal F (G) \big) / \mathsf q \big( \widehat{\mathcal B_K (G)} \big)
\]
is the class group of the embedding $\widehat{\mathcal B_K (G)} \hookrightarrow \mathcal F (G)$.

If $K = G$, then $\widehat{\mathcal B_K (G)} = \mathcal B_K (G) = \mathcal F (G)$, and all the statements hold.
Suppose now that $K$ is a proper subgroup of $G$.
If $|G| = 2$, then $K = \{ 1_G \}$, and so $\widehat{\mathcal B_K (G)} = \mathcal B_K (G) = \mathcal F \big( \{ 1_G, g^{[2]} \} \big)$, with $g \in G \setminus \{ 1_G \}$, is factorial.
However, $g \in G \setminus \{ 1_G \}$ is not the gcd of any finite subset of $\mathcal B_K (G)$, which implies that the embedding is not a divisor theory.
Thus, we may assume that $|G| \ge 3$.
Let $g \in G$.
If $g \in K$, then clearly $g$ is the gcd of the product-$K$ sequence $g \in \mathcal A_K (G) \subseteq \widehat{\mathcal B_K (G)}$.
If $g \in G \setminus K$, then there exists $h \in G \setminus \{ g^{-1}, 1_G \}$, and hence $g$ is the gcd of two product-$K$ sequences $g \bdot g^{-1}$ and $g \bdot h \bdot (g^{-1}h^{-1})$ in $\mathcal F (G)$.
It remains to show that the embedding $\widehat{\mathcal B_K (G)} \hookrightarrow \mathcal F (G)$ is a divisor homomorphism.
Let $S, T \in \widehat{\mathcal B_K (G)}$ with $T \mid S$ in $\mathcal F (G)$.
Since $G/K$ is a torsion group, Proposition~\ref{pro:local}.2 ensures that $\pi (S) \subseteq KG'$ and $\pi (T) \subseteq KG'$.
Thus, $\pi \big( S \bdot T^{[-1]} \big) \subseteq \{ xy^{-1} \mid x \in \pi (S) \mbox{ and } y \in \pi (T) \} \subseteq KG'$, whence $S \bdot T^{[-1]} \in \widehat{\mathcal B_K (G)}$.
Therefore, the embedding $\widehat{\mathcal B_K (G)} \hookrightarrow \mathcal F (G)$ is a divisor theory, and by the equivalent conditions in Section~\ref{2}.2, $\widehat{\mathcal B_K (G)}$ is a Krull monoid with the class group $\mathcal C \big( \widehat{\mathcal B_K (G)} \big) = \mathcal F (G) / \mathcal B_K (G)$.

Now, we define the map
\[
  \Psi \colon \mathcal F (G) / \mathcal B_K (G) \to G/KG'
\]
by $\Psi \big( [S] \big) = g(KG')$ for any $g \in \pi (S)$, where $[S] = S \mathsf q \big( \mathcal B_K (G) \big)$ is a congruence class of $S \in \mathcal F (G)$.

To show that $\Psi$ is well-defined, assume that $[S] = [S']$ for $S, S' \in \mathcal F (G)$.
Then, $S \bdot T = S' \bdot T'$ for some $T, T' \in \mathcal B_K (G)$, and as a $KG'$-coset, we obtain that $\pi (S \bdot T) = \pi (S' \bdot T') \subseteq g(KG') \cap g'(KG')$, where $g \in \pi (S)$ and $g' \in \pi (S')$.
Hence, $g(KG') = g'(KG')$, and thus $\Psi$ is well-defined.

For any $g \in G$, there exists $S \in \mathcal F (G)$ such that $g \in \pi (S)$, and thus $\Psi ( [S] ) = g(KG')$, which implies that $\Psi$ is surjective.

For $S, S' \in \mathcal F (G)$ with $g \in \pi (S)$ and $g' \in \pi (S')$,
\[
  \Psi \big( [S] \big) \Psi \big( [S'] \big) = \big( g(KG') \big) \big( g' (KG') \big) = (gg')(KG') = \Psi \big( [S \bdot S'] \big) = \Psi \big( [S] + [S'] \big) \,,
\]
whence $\Psi$ is a group homomorphism.

To show that $\Psi$ is injective, let $S, S' \in \mathcal F (G)$ with $g \in \pi (S)$ and $g' \in \pi (S')$ such that $g(KG') = g'(KG')$.
Then, there are $n \in \mathbb N$, $x_1, \ldots, x_n, y_1, \ldots, y_n \in G$, and $k \in K$ such that
\[
  (g')^{-1}g = k {\small \prod}_{i \in [1,n]} x_i y_i x^{-1}_i y^{-1}_i \,.
\]
Let $T := \prod^{\bullet}_{i \in [1,n]} \big( x_i \bdot y_i \bdot x^{-1}_i \bdot y^{-1}_i \big) \in \mathcal F (G)$.
Then, we obtain that
\[
  S \bdot \big( S' \bdot g^{-1} \bdot k \bdot T \big) = S' \bdot \big( S \bdot g^{-1} \bdot k \bdot T \big) \in \mathcal F (G) \,.
\]
Note that $1_G \in \pi (T)$ and $k^{-1}(g')^{-1}g \in \pi (T)$.
Hence, $1_G \in \pi \big( S' \bdot g^{-1} \bdot k \bdot T \big)$ and $k \in \pi \big( S \bdot g^{-1} \bdot k \bdot T \big)$, which means that $[S] = [S']$ in $\mathcal F (G) / \mathsf q \big( \mathcal B_K (G) \big)$.

Finally, to show that every class contains prime divisors, let $S \in \mathcal F (G)$ with $g \in \pi (S)$.
If $g \in K$, then $S, g \in \mathcal B_K (G)$, and $S \bdot g = g \bdot S$, and thus $[S] = [g]$.
If $g \in G \setminus K$, then $g \bdot g^{-1} \in \mathcal B (G) \subseteq \mathcal B_K (G)$, and $S \bdot g^{-1} \in \mathcal B (G) \subseteq \mathcal B_K (G)$, which implies that $S \bdot \big( g \bdot g^{-1} \big) = g \bdot \big( S \bdot g^{-1} \big)$.
Hence, we obtain that $[S] = [g]$.
In particular, by the bijectivity of $\Psi$, we obtain that $[S] \cap G = g (KG')$, which means that the class $[S]$ contains precisely $|KG'|$ prime divisors.

\smallskip
2. We first assert that
\[
  \big( \mathcal B_K (G) \colon \widehat{\mathcal B_K (G)} \big) = \big\{ S \in \mathcal F (G) \mid \pi \big( \phi_K (S) \big) = \left( G/K \right)' \big\} \,.
\]

($\subseteq$) Let $S \in \big( \mathcal B_K (G) \colon \widehat{\mathcal B_K (G)} \big)$.
Since $\widehat{\mathcal B_K (G)}$ contains the trivial sequence, it follows that $S \in S \widehat{\mathcal B_K (G)} \subseteq \mathcal B_K (G)$.
Under the epimorphism $\phi_K \colon \mathcal F (G) \to \mathcal F (G/K)$, we obtain that $\phi_K (S) \in \mathcal B (G/K)$, and thus $\pi \big( \phi_K (S) \big) \subseteq (G/K)'$.
Let $gK \in (G/K)'$.
Note that $(G/K)' = KG'/K$ by a basic property in group theory (see, for instance, \cite[Exercise 3.15]{Ro12}).
Thus, $gK = xK$ for some $x \in KG'$, and so $x^{-1}g \in K \subseteq KG'$.
Since $x \in KG'$ and $KG'$ is a group itself, we infer that $g^{-1} \in KG'$.
By Proposition~\ref{pro:local}.2, $g^{-1} \in \widehat{\mathcal B_K (G)}$ as a sequence of length 1, and hence $S \bdot g^{-1} \in S \widehat{\mathcal B_K (G)} \subseteq \mathcal B_K (G)$.
By Lemma~\ref{lem:shift}, there exists a particular ordering $\sigma (S) \in \pi (S)$ such that $\sigma (S) = gk$ for some $k \in K$, which implies that $\sigma \big( \phi_K (S) \big) = \phi_K \big( \sigma (S) \big) = \phi_K (gk) = gK$, whence $gK \in \pi \big( \phi_K (S) \big)$.

($\supseteq$) Let $S \in \mathcal F (G)$ with $\pi \big( \phi_K (S) \big) = (G/K)' = KG'/K$.
We need to verify that $S \bdot T \in \mathcal B_K (G)$ for every $T \in \widehat{\mathcal B_K (G)}$.
Let $T \in \widehat{\mathcal B_K (G)}$.
Then, by Proposition~\ref{pro:local}.2, $\pi (T) \subseteq KG'$, which implies that $\pi \big( \phi_K (T) \big) \subseteq KG'/K = (G/K)'$.
Since $\pi \big( \phi_K (S) \big) = (G/K)'$, we infer that $x^{-1} \in \pi \big( \phi_K (S) \big)$ for every $x \in \pi \big( \phi_K (T) \big)$.
Thus, $\pi \big( \phi_K (S) \big) \pi \big( \phi_K (T) \big)$ contains the identity element of $G/K$, and since $\pi \big( \phi_K (S) \big) \pi \big( \phi_K (T) \big) \subseteq \pi \big( \phi_K (S) \bdot \phi_K (T) \big) = \pi \big( \phi_K (S \bdot T) \big)$, we obtain that $\phi_K (S \bdot T) \in \mathcal B (G/K)$.
Hence, $S \bdot T \in \mathcal B_K (G)$, and since $T$ was an arbitrary element in $\widehat{\mathcal B_K (G)}$, it follows that $S \in \big( \mathcal B_K (G) \colon \widehat{\mathcal B_K (G)} \big)$.

($\Rightarrow$) If $\big( \mathcal B_K (G) \colon \widehat{\mathcal B_K (G)} \big) \neq \emptyset$, take $S \in \mathcal F (G)$ with $\pi \big( \phi_K (S) \big) = (G/K)'$.
Since $S$ is a finite collection of terms from $G$, $\pi \big( \phi_K (S) \big)$ must be finite.

($\Leftarrow$) Let $(G/K)' = \{ g_1K, \ldots, g_{\ell}K \}$ be the finite commutator subgroup of $G/K$.
For each $i \in [1,\ell]$, $g_iK = \prod_{k \in [1,n_i]} (x_k K) (y_k K) (x^{-1}_k K) (y^{-1}_k K)$ for some $x_1, \ldots, x_{n_i}, y_1, \ldots, y_{n_i} \in G$.
Then, we consider a sequence
\[
  S := {\small \prod}^{\bullet}_{i \in [1,\ell]} \left( {\small \prod}^{\bullet}_{k \in [1,n_i]} x_k \bdot y_k \bdot x^{-1}_k \bdot y^{-1}_k \right) \in \mathcal F (G) \,,
\]
and we clearly obtain that $\pi \big( \phi_K (S) \big) = (G/K)'$, whence $S \in \big( \mathcal B_K (G) \colon \widehat{\mathcal B_K (G)} \big)$.

\smallskip
3. (a) $\Rightarrow$ (b) Suppose that $G' \subseteq K$.
Then, $K \subseteq KG' \subseteq KK \subseteq K$, and so $KG' = K$.
Moreover, since $G/K$ is abelian, we infer that, for every $S \in \mathcal F (G)$, $\pi (S)$ is contained in a $K$-coset.
Since $G/K$ is a torsion group, Proposition~\ref{pro:local}.2 ensures that $\widehat{\mathcal B_K (G)} = \{ S \in \mathcal F (G) \mid \pi (S) \subseteq K \} = \mathcal B_K (G)$, and by Item 1, we infer that $\mathcal B_K (G)$ is a Krull monoid.

\smallskip
(b) $\Leftrightarrow$ (c) $\Leftrightarrow$ (d) follows by Proposition~\ref{pro:local}.2.

\smallskip
(d) $\Rightarrow$ (g) Let $S, T \in \mathcal B_K (G)$ with $T \mid S$ in $\mathcal F (G)$.
Then, $S = T \bdot W$ for some $W = g_1 \bdot \ldots \bdot g_{\ell} \in \mathcal F (G)$.
Since $G/K$ is a torsion group, we can define $n = \lcm \left\{ \ord_K (g_i) \mid i \in [1,\ell] \right\}$.
Hence, $\big( S \bdot T^{[-1]} \big)^{[n]} = W^{[n]} \in \mathcal B_K (G)$, and since $\mathcal B_K (G)$ is root-closed, we infer that $W = S \bdot T^{[-1]} \in \mathcal B_K (G)$.
Thus, $T \mid S$ in $\mathcal B_K (G)$.

\smallskip
(g) $\Rightarrow$ (b) follows from the equivalent conditions for Krull monoids (see Section~\ref{2}.2).

\smallskip
(b) $\Rightarrow$ (e) and (f) Obvious.

\smallskip
(e) $\Rightarrow$ (a) Suppose that $\mathcal B_K (G)$ is a transfer Krull monoid.
Then, there exists a subset $G^{*}_0$ of an abelian group $G^{*}$ such that $\varphi \colon \mathcal B_K (G) \to \mathcal B (G^{*}_0)$ is a transfer homomorphism.
To show that $G/K$ is abelian, we assume to the contrary that $G/K$ is non-abelian.
Then, there exist $gK, hK \in G/K$ with $(gK)(hK) \neq (hK)(gK)$, equivalently $ghg^{-1}h^{-1} \notin K$.
We proceed with the following assertion.

\smallskip
\begin{itemize}
\item[\namedlabel{itm:A}{\bf A}.] $S := h \bdot h^{-1} \bdot g \bdot (hg^{-1}h^{-1}) \in \mathcal A_K (G)$.
\end{itemize}

\begin{proof}[Proof of \ref{itm:A}]
It is clear that $S \in \mathcal B (G) \subseteq \mathcal B_K (G)$.
Assume to the contrary that $S \notin \mathcal A_K (G)$.
Then, $S = S_1 \bdot S_2$ for some non-trivial sequences $S_1, S_2 \in \mathcal B_K (G)$ with $h \mid S_1$.
Since $gK$ and $hK$ are non-commuting elements, we may further assume that $g, h \notin K$.
This allows us to suppose that $|S_1| = |S_2| = 2$.
Since $ghg^{-1}h^{-1} \notin K$, we obtain that $S_1 \neq h \bdot h^{-1}$.
If $S_1 = h \bdot g$, then $h = g^{-1}k$ for some $k \in K$, which ensures that $ghg^{-1}h^{-1} \in K$, a contradiction.
If $S_1 = h \bdot (hg^{-1}h^{-1})$, then $S_2 = h^{-1} \bdot g$, and so $g = hk$ for some $k \in K$, which again leads to a contradiction.
\qedhere{[\ref{itm:A}]}
\end{proof}

\smallskip
By \ref{itm:A}, we obtain that $\varphi (S)$ is a minimal product-one sequence in $\mathcal B (G^{*}_0)$.
Since $G/K$ is a torsion group, we can define $n = \ord_K (ghg^{-1}h^{-1}) \ge 2$, and then $S^{[n]} = \big( h \bdot h^{-1} \big)^{[n]} \bdot \big( g \bdot (hg^{-1}h^{-1}) \big)^{[n]}$, and
\begin{equation} \label{eq:divide}~
  \varphi (S)^{[n]} = \varphi \big( S^{[n]} \big) = \varphi \big( h \bdot h^{-1} \big)^{[n]} \bdot \varphi \Big( \big( g \bdot (hg^{-1}h^{-1}) \big)^{[n]} \Big) \,.
\end{equation}
Since $G^{*}$ is abelian, the embedding $\mathcal B (G^{*}_0) \hookrightarrow \mathcal F (G^{*}_0)$ is a divisor homomorphism.
Thus, in view of (\ref{eq:divide}), we obtain that $\varphi \big( h \bdot h^{-1} \big)^{[n]}$ divides $\varphi (S)^{[n]}$ in $\mathcal B (G^{*}_0)$.
Since $\mathcal B (G^{*}_0)$ is root-closed, we infer that $\varphi \big( h \bdot h^{-1} \big)$ divides $\varphi (S)$ in $\mathcal B (G^{*}_0)$.
Note that $\varphi \big( h \bdot h^{-1} \big)$ and $\varphi (S)$ are both atoms in $\mathcal B (G^{*}_0)$.
It follows that $\varphi \big( h \bdot h^{-1} \big) = \varphi (S)$, which ensures that $\varphi \Big( \big( g \bdot (hg^{-1}h^{-1}) \big)^{[n]} \Big)$ is a trivial sequence in $\mathcal B (G^{*}_0)$, a contradiction.

\smallskip
(f) $\Rightarrow$ (b) Suppose that $\mathcal B_K (G)$ is a weakly Krull monoid.
Since $G/K$ is a torsion group, Proposition~\ref{pro:local}.2 ensures that
\[
  \widehat{\mathcal B_K (G)} = \{ S \in \mathcal F (G) \mid \pi (S) \subseteq KG' \} = \small{\bigcap}_{\mathfrak p \in \mathfrak X (\mathcal B_K (G))} \mathcal B_K (G)_{\mathfrak p} = \mathcal B_K (G) \,,
\]
whence the assertion follows from Item 1.

\smallskip
(g) $\Leftrightarrow$ (h) Since $G/K$ is a torsion group, the embedding $\mathcal B_K (G) \hookrightarrow \mathcal F (G)$ is cofinal, as shown in the proof of Item 1.
Then, the map $\theta \colon \mathcal C \big( \mathcal B_K (G), \mathcal F (G) \big) \to \mathcal F (G) / \mathcal B_K (G)$, given by $\theta \big( [S]^{\mathcal F (G)}_{\mathcal B_K (G)} \big) = S \mathsf q \big( \mathcal B_K (G) \big)$, is a well-defined group epimorphism, and the equivalent conditions follow from \cite[Proposition 2.8.7]{Ge-HK06}.
\end{proof}

\smallskip
Let $H$ be an atomic monoid.
For two factorizations $z, z'$ in $\mathsf Z (H)$, we can write
\[
  z = u_1 \cdot \ldots \cdot u_{\ell} \cdot v_1 \cdot \ldots \cdot v_m \quad \mbox{ and } \quad z' = u_1 \cdot \ldots \cdot u_{\ell} \cdot w_1 \cdot \ldots \cdot w_n
\]
for $\ell, m, n \in \mathbb N_0$ and $u_1, \ldots, u_{\ell}, v_1, \ldots, v_m, w_1, \ldots, w_n \in \mathcal A \big( H_{\text{red}} \big)$ with $\{ v_1, \ldots, v_m \} \cap \{ w_1, \ldots, w_n \} = \emptyset$.
Then, we denote by $\mathsf d (z, z') = \max \{ m, n \} \in \mathbb N_0$ the {\it distance} between $z$ and $z'$.
For $a \in H$, we define $\mathsf c (a)$ to be the smallest integer $N \in \mathbb N_0 \cup \{ \infty \}$ such that any two factorizations $z, z' \in \mathsf Z (a)$ can be concatenated by an $N$-chain, i.e., there exists a finite sequence of factorizations $z = z_0, z_1, \ldots, z_k = z'$ in $\mathsf Z (a)$ such that $\mathsf d (z_{i-1}, z_i) \le N$ for every $i \in [1,k]$.
Then, we denote by $\mathsf c (H) = \sup \{ \mathsf c (a) \mid a \in H \} \in \mathbb N_0 \cup \{ \infty \}$ the {\it catenary degree} of $H$.
By definition, it is easy to see that $H$ is factorial if and only if $\mathsf c (H) = 0$.
The catenary degree is another well-understood arithmetic invariant for Krull monoids, in particular for the monoid $\mathcal B (G)$ when $G$ is abelian (see \cite{Ge-Gr-Sc11,Ge-Zh15}), and it has also recently been studied in the non-abelian setting (see \cite{Ge-Gr-Oh-Zh22,Oh26b}).

Let $\theta \colon H \to D$ be a transfer homomorphism of atomic monoids.
Then, $\theta$ is naturally extended to a surjective homomorphism $\overline{\theta} \colon \mathsf Z (H) \to \mathsf Z (D)$ satisfying $\overline{\theta} (uH^{\times}) = \theta (u) D^{\times}$ for every $u \in \mathcal A (H)$ (see \cite[Lemma 1.4.2]{Ge-Gr-Zh26} for more details).
For $a \in H$, $\mathsf c (a, \theta)$ denotes the smallest integer $N \in \mathbb N_0 \cup \{ \infty \}$ with the following property:
\begin{itemize}
\item[] If $z, z' \in \mathsf Z (a)$ with $\overline{\theta} (z) = \overline{\theta} (z')$, then there exist $k \in \mathbb N_0$ and factorizations $z = z_0, z_1, \ldots, z_k = z'$ in $\mathsf Z (a)$ such that $\overline{\theta} (z_i) = \overline{\theta} (z)$ and $\mathsf d (z_{i-1}, z_i) \le N$ for every $i \in [1,k]$.
\end{itemize}
We denote by $\mathsf c (H, \theta) = \sup \{ \mathsf c (a, \theta) \mid a \in H \} \in \mathbb N_0 \cup \{ \infty \}$ the {\it catenary degree in the fibres of $\theta$}.
Then, it is well-known that (see \cite[Lemma 1.4.4]{Ge-Gr-Zh26})
\begin{equation} \label{eq:cat}~
  \mathsf c (D) \le \mathsf c (H) \le \max \{ \mathsf c (D), \mathsf c (H, \theta) \} \,.
\end{equation}

The following theorem is the main result on arithmetic properties of $\mathcal B_K (G)$.
If $G$ is abelian with $|G| \ge 3$, then Theorem~\ref{thm:struc} ensures that $\mathcal B_K (G)$ is a Krull monoid with the class group $G/K$, and thus the transfer result in the following theorem just follows directly from the classical result for (general) Krull monoids described in Proposition~\ref{pro:Krull}.
If $G$ is not necessarily abelian but $K$ is trivial, then $\mathcal B_K (G) \cong \mathcal B (G/K)$, and so this isomorphism is a transfer homomorphism.
The next theorem shows that this remains true in the general case.
In particular, if $G/K$ is finite, then we will see that $\mathcal B_K (G)$ is a C-monoid (by Corollary~\ref{coro:seminormal}), and it admits a transfer homomorphism analogous to that in the general Krull case of Proposition~\ref{pro:Krull}.

\smallskip
\begin{theorem} \label{thm:transfer}~
Let $G$ be a group and $K \unlhd G$.
\begin{enumerate}
\item The canonical map $\mathcal B_K (G) \to \mathcal B \big( G/K \big)$ is a transfer homomorphism. Thus, $\mathcal L \big( \mathcal B_K (G) \big) = \mathcal L \big( \mathcal B (G/K) \big)$ and all results, known for length sets of product-one sequences over $G/K$, carry over to $\mathcal B_K(G)$. In particular, we have

	\begin{enumerate}
	\smallskip
	\item[(i)] $\mathcal U_k \big( \mathcal B_K (G) \big) = \mathcal U_k \big( \mathcal B (G/K) \big)$ is always an interval for every $k \in \mathbb N$,
	
	\smallskip
	\item[(ii)] $\mathsf D_K (G) = \mathsf D \big( G/K \big)$, $\rho_k \big( \mathcal B_K (G) \big) = \rho_k \big( \mathcal B (G/K) \big) \le k \max \big\{ 1, \frac{\mathsf D (G/K)}{2} \big\} = k \max \big\{ 1, \frac{\mathsf D_K (G)}{2} \big\}$ for every $k \in \mathbb N$, and $\rho \big( \mathcal B_K (G) \big) = \rho \big( \mathcal B (G/K) \big) \le \max \big\{ 1, \frac{\mathsf D (G/K)}{2} \big\} = \max \big\{ 1, \frac{\mathsf D_K (G)}{2} \big\}$.
		     In particular, if $G/K$ is finite, then $\mathsf D_K (G)$ is finite, and $\rho \big( \mathcal B_K (G) \big) = \max \big\{ 1, \frac{\mathsf D_K (G)}{2} \big\}$ is finite and accepted.

	\end{enumerate}
	
\smallskip
\item The class semigroup of $\mathcal B_K (G)$ in $\mathcal F (G)$ is isomorphic to the class semigroup of $\mathcal B (G/K)$ in $\mathcal F (G/K)$.

\smallskip
\item The following statements are equivalent:
	\begin{enumerate}
	\smallskip
	\item[(a)] $\mathcal B_K (G)$ is half-factorial.

	\smallskip
	\item[(b)] $|G/K| \le 2$.
	
	\smallskip
	\item[(c)] $\mathcal B \big( G/K \big)$ is factorial.
	
	\smallskip
	\item[(d)] $\mathcal B \big( G/K \big)$ is half-factorial.
	\end{enumerate}

\smallskip
\item The following statements are equivalent:
	\begin{enumerate}
	\smallskip
	\item[(a)] $\mathcal B_K (G)$ is factorial.
	
	\smallskip
	\item[(b)] $|G/K| \le 2$ and $|G \setminus K| \le1$.
	
	\smallskip
	\item[(c)] Either $K = G$, or else $G \cong C_2$ (a cyclic group of order 2), and $K = \{ 1_G \}$.
	\end{enumerate}
	\medskip
	In particular, if we take $K = \{ 1_G \}$, then 3.(b) is equivalent to 4.(b), and hence all statements in Items 3 and 4 are equivalent.

\medskip
\item We have $\mathsf c \big( \mathcal B_K (G), \phi_K \big) \le 2$, and  if $|G/K| \ge 3$, then $\mathsf c \big( \mathcal B_K (G) \big) = \mathsf c \big( \mathcal B (G/K) \big)$.
\end{enumerate}
\end{theorem}

\begin{proof}
1. If $K = G$, then $\mathcal B_K (G) = \mathcal F (G)$ is factorial with trivial class group, and all claims follow from the classical result Proposition~\ref{pro:Krull}.
Now, in view of (\ref{eq:phi}), recall that, for every $S \in \mathcal F (G)$, $S \in \mathcal B_K (G)$ if and only if $\phi_K (S) \in \mathcal B (G/K)$.
Thus, we infer that $\phi_K \big( \mathcal B_K (G) \big) = \mathcal B (G/K)$, so that the restriction of $\phi_K \big|_{\mathcal B_K (G)} \colon \mathcal B_K (G) \to \mathcal B (G/K)$ is surjective and $\phi^{-1}_K  \big( 1_{\mathcal F (G/K)} \big) = \{ 1_{\mathcal F (G)} \}$.
To conclude that $\phi_K \big|_{\mathcal B_K (G)}$ is a transfer homomorphism, let $S = g_1 \bdot \ldots \bdot g_{\ell} \in \mathcal B_K (G)$.
Suppose that $\phi_K (S) = T \bdot W$ with $T, W \in \mathcal B (G/K)$.
By renumbering if necessary, we may assume that $T = \phi_K (g_1) \bdot \ldots \bdot \phi_K (g_t)$ and $W = \phi_K (g_{t+1}) \bdot \ldots \bdot \phi_K (g_{\ell})$ for some $t \in [1, \ell]$.
Then, $g_1 \bdot \ldots \bdot g_t = \phi^{-1}_K (T) \in \mathcal B_K (G)$ and $g_{t+1} \bdot \ldots \bdot g_{\ell} = \phi^{-1}_K (W) \in \mathcal B_K (G)$, and we obviously obtain that $S = \phi^{-1}_K (T) \bdot \phi^{-1}_K (W)$.
Thus, the restriction of $\phi_K$ to $\mathcal B_K (G)$ is a transfer homomorphism to $\mathcal B (G/K)$, which ensures that $\mathcal L \big( \mathcal B_K (G) \big) = \mathcal L \big( \mathcal B (G/K) \big)$.
Moreover, all remaining statements follow by Lemma~\ref{lem:trans}:

\smallskip
(i) For every $k \in \mathbb N$, we obtain that $\mathcal U_k \big( \mathcal B_K (G) \big) = \mathcal U_k \big( \mathcal B (G/K) \big)$, and the latter is always an interval. This is a well-known result for the monoid of product-one sequences; see \cite[Theorem 4.4]{Fa-Zh23}.
More precisely, these are all finite intervals if $G/K$ is finite, whereas, if $G/K$ is infinite, then $\mathcal U_k \big( \mathcal B_K (G) \big) = \mathcal U_k \big( \mathcal B (G/K) \big) = \mathbb N_{\ge 2}$ for every $k \ge 2$.

\smallskip
(ii) If $K = G$, then both $\mathcal B_K (G) = \mathcal F (G)$ and $\mathcal B (G/K) \cong \mathcal F \big( \{ 1_G \} \big)$ are factorial monoids, and hence $\mathsf D_K (G) = 1 = \mathsf D (G/K)$, $\rho_{k} \big( \mathcal B_K (G) \big) = k = \rho_{k} \big( \mathcal B (G/K) \big)$, and $\rho \big( \mathcal B_K (G) \big) = 1 = \rho \big( \mathcal B (G/K) \big)$.
Thus, we suppose that $K \neq G$.
Since $\phi_K \big( \mathcal A_K (G) \big) = \mathcal A (G/K)$, we obtain that
\[
  \mathsf D_K (G) = \sup \big\{ |S| \mid S \in \mathcal A_K (G) \big\} = \sup \big\{ |T| \mid T \in \phi_K \big( \mathcal A_K (G) \big) \big\} = \mathsf D (G/K) \,.
\]
Moreover, since $\mathsf L (S) = \mathsf L \big( \phi_K (S) \big)$ for every $S \in \mathcal B_K (G)$, it follows that, for every $k \in \mathbb N$,
\begin{equation} \label{eq:rho_k}~
  \rho_k \big( \mathcal B_K (G) \big) = \rho_k \big( \mathcal B (G/K) \big) \le \frac{k \mathsf D (G/K)}{2} = \frac{k \mathsf D_K (G)}{2}
\end{equation}
where the inequality follows from \cite[Proposition 4.1]{Fa-Zh23}.
Thus, in view of (\ref{eq:rho}), we also obtain that
\begin{equation} \label{eq:rhoBK}~
  \rho \big( \mathcal B_K (G) \big) = \rho \big( \mathcal B (G/K) \big) \le \frac{\mathsf D (G/K)}{2} = \frac{\mathsf D_K (G)}{2} \,.
\end{equation}

Suppose now that $G/K$ is finite.
Then, $\mathsf D_K (G) = \mathsf D (G/K) \le |G/K|$ is finite (see \cite[Lemma 2.4]{Ge-Gr13}), and moreover $\rho \big( \mathcal B (G/K) \big) = \frac{\mathsf D (G/K)}{2}$ by \cite[Proposition 5.6]{Oh20}, which implies that equality holds in (\ref{eq:rhoBK}).
Hence, we conclude that $\rho \big( \mathcal B_K (G) \big) = \frac{\mathsf D_K (G)}{2}$ is finite.
Let $U = g_1 \bdot \ldots \bdot g_{\ell} \in \mathcal A_K (G)$ with $\ell = \mathsf D_K (G)$.
Then, $U^{-1} := g^{-1}_1 \bdot \ldots \bdot g^{-1}_{\ell} \in \mathcal A_K (G)$, and so $\{ 2, \mathsf D_K (G) \} \subseteq \mathsf L \big( U \bdot U^{-1} \big)$.
Since $\max \mathsf L \big( U \bdot U^{-1} \big) \le \rho_2 \big( \mathcal B_K (G) \big) \le \mathsf D_K (G)$, where the second inequality follows from (\ref{eq:rho_k}), it follows that $\rho \big( \mathsf L ( U \bdot U^{-1} ) \big) = \frac{\mathsf D_K (G)}{2}$, whence $\rho \big( \mathcal B_K (G) \big)$ is accepted.

\smallskip
2. Define the map $\theta \colon \mathcal C \big( \mathcal B_K (G), \mathcal F (G) \big) \to \mathcal C \big( \mathcal B (G/K), \mathcal F (G/K) \big)$ by $\theta \left( [S]^{\mathcal F (G)}_{\mathcal B_K (G)} \right) = \left[ \phi_K (S) \right]^{\mathcal F (G/K)}_{\mathcal B (G/K)}$ for every $S \in \mathcal F (G)$.
To show that $\theta$ is well-defined, let $S_1 \sim_{\mathcal B_K (G)} S_2$ in $\mathcal F (G)$ for $S_1, S_2 \in \mathcal F (G)$.
Let $T' \in \mathcal F (G/K)$.
Since $\phi_K$ is a surjective homomorphism, there exists $T \in \mathcal F (G)$ such that $\phi_K (T) = T'$, and $\phi_K (S_i) \bdot \phi_K (T) = \phi_K (S_i \bdot T)$ for all $i \in [1,2]$.
Thus, we obtain that
\[
  \begin{aligned}
    \phi_K (S_1) \bdot T' = \phi_K (S_1 \bdot T) \in \mathcal B (G/K) \quad & \Leftrightarrow \quad S_1 \bdot T \in \mathcal B_K (G) \\
    									     & \Leftrightarrow \quad S_2 \bdot T \in \mathcal B_K (G) \\
									     & \Leftrightarrow \quad \phi_K (S_2) \bdot T' = \phi_K (S_2 \bdot T) \in \mathcal B (G/K) \\
  \end{aligned}
\]
where the second equivalence follows from $S_1 \sim_{\mathcal B_K (G)} S_2$ in $\mathcal F (G)$.
This implies that $\phi_K (S_1) \sim_{\mathcal B (G/K)} \phi_K (S_2)$ in $\mathcal F (G/K)$, which implies that $\theta$ is well-defined.
Clearly, $\theta$ is a surjective semigroup homomorphism.
To show injectivity, let $S_1, S_2 \in \mathcal F (G)$ with $\phi_K (S_1) \sim_{\mathcal B (G/K)} \phi_K (S_2)$ in $\mathcal F (G/K)$.
For every $T \in \mathcal F (G)$,
\[
  \begin{aligned}
    S_1 \bdot T \in \mathcal B_K (G) \quad & \Leftrightarrow \quad \phi_K (S_1) \bdot \phi_K (T) = \phi_K (S_1 \bdot T) \in \mathcal B (G/K) \\
                                                                  & \Leftrightarrow \quad \phi_K (S_2 \bdot T) = \phi_K (S_2) \bdot \phi_K (T) \in \mathcal B (G/K) \\
                                                                  & \Leftrightarrow \quad S_2 \bdot T \in \mathcal B_K (G) \,,
  \end{aligned}
\]
where the second equivalence follows from $\phi_K (S_1) \sim_{\mathcal B (G/K)} \phi_K (S_2)$ in $\mathcal F (G/K)$.
This implies that $S_1 \sim_{\mathcal B_K (G)} S_2$ in $\mathcal F (G)$, whence $\theta$ is injective.

\smallskip
3. (b) $\Rightarrow$ (c) $\Rightarrow$ (d) $\Rightarrow$ (b) are classical results, and (a) $\Leftrightarrow$ (d) follows by Item 1 and Lemma~\ref{lem:trans}.3.

\smallskip
4. (a) $\Rightarrow$ (b) Since every factorial monoid is half-factorial, Item 3 shows that $|G/K| \le 2$.
If $|G/K| = 1$, then $K = G$ and the assertion follows.
Suppose that $|G/K| = 2$, so that $K \lneq G$.
We proceed with the following assertion.

\smallskip
\begin{itemize}
\item[\namedlabel{itm:AA}{\bf A}.] $\mathcal A_K (G) = K \cup \{ g_1 \bdot g_2 \mid g_1, g_2 \in G \setminus K \}$.
\end{itemize}

\begin{proof}[Proof of \ref{itm:AA}]
($\supseteq$) Clearly, every element in $K$ is an atom in $\mathcal B_K (G)$ of length 1.
Let $g_1, g_2 \in G\setminus K$.
Then, since $|G/K| = 2$, it follows that $g_1 K = g_2 K$ in $G/K$ and $(g_1 g_2)K = \big( g_1 K \big) \big( g_2 K \big) = \big( g_1 K \big)^{2} = K$, and hence $g_1 \bdot g_2$ is an atom in $\mathcal B_K (G)$ of length 2.

($\subseteq$) Let $S = g_1 \bdot \ldots \bdot g_{\ell} \in \mathcal A_K (G)$.
If $\ell = 1$, then since $S$ is product-$K$, it follows that $S = g_1 \in K$.
Suppose now that $\ell \ge 2$.
If $\supp (S) \cap K \neq \emptyset$, then there exists some $i \in [1,\ell]$ such that $g_i \in K$.
Since $S$ is product-$K$, Lemma~\ref{lem:shift} ensures that $g_i \sigma \big( S \bdot g^{[-1]}_i \big) \in K$ for some $\sigma \big( S \bdot g^{[-1]}_i \big) \in \pi \big( S \bdot g^{[-1]}_i \big)$, whence $\sigma \big( S \bdot g^{[-1]}_i \big) \in K$.
This implies that $S = g_i \bdot \big( S \bdot g^{[-1]}_i \big)$ is not an atom in $\mathcal B_K (G)$, and thus we further assume that $\supp (S) \cap K = \emptyset$.
Now, by passing to $G/K$, we obtain that $\big( g_1 K \big) \bdot \ldots \bdot \big( g_{\ell} K \big) \in \mathcal A \big( G/K \big)$.
Since $\supp (S) \cap K = \emptyset$, it follows that $g_1 K = \cdots = g_{\ell} K$.
Since $|G/K| = 2$, $\big( g_1 K \big) \bdot \ldots \bdot \big( g_{\ell} K \big)$ being an atom over $G/K$ must ensure that $\ell \le 2$.
Thus, $\ell = 2$, and since $S \in \mathcal A_K (G)$, it follows that $S = g_1 \bdot g_2$ with $g_1, g_2 \in G \setminus K$.
\qedhere{[\ref{itm:AA}]}
\end{proof}

If $|G \setminus K| \ge 2$, then in view of \ref{itm:AA}, there exist $g_1, g_2 \in G \setminus K$ with $g_1 \neq g_2$ such that $g_1 \bdot g_2 \in \mathcal A_K (G)$, and thus $\big( g_1 \bdot g_2 \big)^{[2]} = \big( g^{[2]}_1 \big) \bdot \big( g^{[2]}_2 \big)$, contradicting that $\mathcal B_K (G)$ is factorial.

\smallskip
(b) $\Rightarrow$ (c) Suppose that $K$ is a proper normal subgroup of $G$.
Then, $|G/K| = 2$ and $|G \setminus K | = 1$, whence $|G| = |K| + 1$.
Let $g \in G \setminus K$.
Since $|G/K| = 2$, it follows that $G = K \cup gK$ with $K \cap gK = \emptyset$.
Therefore, $|K| = |gK| = |G \setminus K| = 1$, and so $|G| = 2$.

\smallskip
(c) $\Rightarrow$ (a) If $K = G$, then $\mathcal B_K (G) = \mathcal F (G)$ is factorial.
If $G \cong C_2$ and $K = \{ 1_G \}$, then $\mathcal B_K (G) \cong \mathcal B (C_2) = \mathcal F \big( \{ 1_G, g^{[2]} \} \big)$, with $g \in G \setminus K$, is factorial.

\smallskip
5. Let $S\in \mathcal{B}_K(G)$ and $z, z'$ be two factorizations of $S$ with $\overline{\phi}_K(z)=\overline{\phi}_K(z')$.
We assert that there exists a 2-chain $z=z_0,z_1,\ldots,z_k=z'$ of factorizations of $S$ such that $\overline{\phi}_K(z)=\overline{\phi}_K(z_i)$ for all $i\in [1,k]$.
Let
\[
  z = U_1\bdot \ldots \bdot U_m \quad \mbox{ and } \quad z' = U'_1\bdot \ldots \bdot U'_n
\]
with $U_i, U_j' \in \mathcal{A}_K(G)$ for all $i\in [1,m]$ and $j\in [1,n]$.
The equality $\overline{\phi}_K(z)=\overline{\phi}_K(z')$ implies that $m=n$ and, after renumbering if necessary, we may assume that $\phi_K(U_i)=\phi_K(U'_i)$ for all $i\in [1,m]$.
Now, we set
\[
  s_i := |\phi_K(U_i)| = |\phi_K(U'_i)| \,\, \mbox{ for all } \,\, i \in [1,m] \quad \mbox{ and } \quad \ell := {\small \sum}_{i \in [1,m]} |\phi_K(U_i)| = {\small \sum}_{i \in [1,m]}|\phi_K(U'_i)| \,.
\]
For each $j \in [1,m]$, let $I_j:=[s_0+ \cdots +s_{j-1}+1,s_0+ \cdots +s_{j-1}+s_j]$ with $s_0:=0$, so that
\[
  [1,\ell] = I_1 \mathbin{\dot{\cup}} \cdots \mathbin{\dot{\cup}} I_m
\]
is a disjoint union of $[1,\ell]$ into $m$ consecutive intervals with $|I_j| = s_j$ for $j \in [1,m]$.
Then, if we write $S = g_1 \bdot \ldots \bdot g_{\ell}$ with $g_1,\ldots, g_\ell\in G$, then $S = U_1\bdot \ldots \bdot U_m = U'_1\bdot \ldots \bdot U'_m$ with
\begin{equation} \label{eq:permutation}~
  U_j = {\small \prod}^{\bullet}_{i\in I_j}g_i \quad \mbox{ and } \quad U'_j = {\small \prod}^{\bullet}_{i\in I_j}g_{\sigma(i)} \,\, \mbox{ for } \,\, j \in [1,m] \,,
\end{equation}
where $\sigma$ is some permutation of $[1,\ell]$.
Since
\[
  {\small \prod}^{\bullet}_{i\in I_j}g_iK=\phi_K(U_j)=\phi_K(U'_j) = {\small \prod}^{\bullet}_{i\in I_j}g_{\sigma(i)}K \,\, \mbox{ for all } \,\, j \in [1,m] \,,
\]
we may assume by renumbering the images of $\sigma$ within each block $I_j$ if necessary that
\begin{equation} \label{eq:sigma}~
  g_i K = g_{\sigma(i)} K \,\, \mbox{ for all } \,\, i \in [1,\ell] \,.
\end{equation}
Accordingly, we may choose the permutation $\sigma$ such that the number of indices $i \in [1,\ell]$ with $\sigma(i) = i$ is maximal among all permutations on $[1,\ell]$ satisfying both (\ref{eq:permutation}) and (\ref{eq:sigma}).
If $\sigma (i) = i$ for all $i \in [1,\ell]$, then $U_j = U'_j$ for all $j \in [1,m]$, and hence $z = z'$, which is the desired 2-chain of factorizations of $S$ with $\overline{\phi}_K (z) = \overline{\phi}_K (z')$.
Thus, we suppose that $\sigma (s) \neq s$ for some $s \in [1, \ell]$, which ensures that $s \in I_{j_1}$ and $\sigma (s) \in I_{j_2}$ for some $j_1, j_2\in [1,m]$.
If $j_1=j_2$, then we can construct a permutation $\tau$ on $[1,\ell]$ such that $\tau (s) = \sigma \big( \sigma (s) \big)$, $\tau \big( \sigma (s) \big) = \sigma (s)$, and $\tau (i) = \sigma (i)$ for all $i \in [1,\ell] \setminus \{ s, \sigma (s) \}$.
Then, in view of (\ref{eq:sigma}), $g_s K = g_{\sigma (s)} K = g_{\sigma (\sigma (s))} K = g_{\tau (s)} K$, and so we infer that $\tau$ satisfies (\ref{eq:sigma}).
We observe that $U'_j = \prod^{\bullet}_{i \in I_j} g_{\sigma (i)} = \prod^{\bullet}_{i \in I_j} g_{\tau (i)}$ for every $j \in [1,m] \setminus \{ j_1 \}$.
Since $s, \sigma (s) \in I_{j_1}$, it follows that
\[
  U'_{j_1} = \left( {\small \prod}^{\bullet}_{i \in I_{j_1} \setminus \{ s, \sigma (s) \}} g_{\sigma (i)} \right) \bdot g_{\sigma (s)} \bdot g_{\sigma(\sigma(s))} = \left( {\small \prod}^{\bullet}_{i \in I_{j_1} \setminus \{ s, \sigma (s) \}} g_{\tau (i)} \right) \bdot g_{\tau(\sigma(s))} \bdot g_{\tau (s)} = {\small \prod}^{\bullet}_{i \in I_{j_1}} g_{\tau (i)} \,,
\]
whence we infer that $\tau$ satisfies (\ref{eq:permutation}).
Since $\sigma (s) \neq s$, it follows that $\sigma \big( \sigma (s) \big) \neq \sigma (s)$, which means that $s, \sigma (s) \notin \{ i \in [1,\ell] \mid \sigma (i) = i \}$.
By construction, the number of indices $i \in [1,\ell]$ with $\tau (i) = i$ is strictly greater than that for $\sigma$, contradicting the maximality of $\sigma$.
Hence, we must have that $j_1\ne j_2$, and we consider the following sequences:
\[
  V_{j_1} := U_{j_1} \bdot (g_s)^{[-1]} \bdot g_{\sigma(s)} \quad \mbox{ and } \quad V_{j_2} := U_{j_2} \bdot (g_{\sigma(s)})^{[-1]} \bdot g_s \,.
\]
Since $\phi_K (V_{j_1}) =  \phi_K (U_{j_1}) \bdot ( g_s K)^{[-1]} \bdot g_{\sigma (s)} K$, in view of (\ref{eq:sigma}), $\phi_K (V_{j_1}) = \phi_K (U_{j_1})$ and $\phi_K (V_{j_1}) \in \mathcal B (G/K)$, equivalently $V_{j_1} \in \mathcal B_K (G)$.
By symmetry, we obtain that $\phi_K (V_{j_2}) = \phi_K (U_{j_2})$ and $V_{j_2} \in \mathcal B_K (G)$.
Since $\phi_K$ is a transfer homomorphism, Lemma~\ref{lem:trans}.1 ensures that $V_{j_1}, V_{j_2} \in \mathcal A_K (G)$.
This allows us to obtain the factorization of $S$
\[
  z_1 := V_1 \bdot \ldots \bdot V_m \,,
\]
where $V_{j_1} := U_{j_1} \bdot (g_s)^{[-1]} \bdot g_{\sigma(s)}$, $V_{j_2} := U_{j_2} \bdot (g_{\sigma(s)})^{[-1]} \bdot g_s$, and $V_j:=U_j$ for all $j \in [1,m] \setminus \{j_1,j_2\}$.
Moreover, by construction, we have that $\overline{\phi}_K (z) = \overline{\phi}_K (z_1)$ and $\mathsf d (z, z_1) \le 2$.
If we define the permutation $\sigma_1$ for $z_1$ in the same way that $\sigma$ was defined for $z$, then the number of indices $i\in [1,\ell]$ with $\sigma(i) = i$ has increased by at least one, because every index fixed by $\sigma$ remains also fixed by $\sigma_1$, and additionally $\sigma_1(s)=s$.
Repeating this process, we obtain the desired 2-chain $z=z_0, z_1, \ldots, z_k = z'$ of factorizations of $S$ with $\overline{\phi}_K (z) = \overline{\phi}_K (z_i)$ for all $i \in [1,k]$, whence we conclude that $\mathsf c \big( \mathcal B_K (G), \phi_K \big) \le 2$.
The remaining assertion follows from (\ref{eq:cat}) and Items 3-4.
\end{proof}

\smallskip
In Theorem~\ref{thm:struc}.3, we established when $\mathcal B_K (G)$ is root-closed or completely integrally closed.
As an immediate consequence of the transfer result, we obtain the following characterization of when $\mathcal B_K (G)$ is a seminormal C-monoid.

\smallskip
\begin{corollary} \label{coro:seminormal}~
Let $G$ be a group and $K \unlhd G$.
\begin{enumerate}
\item The following statements are equivalent:
      \begin{enumerate}
      \smallskip
      \item[(a)] $\mathcal B_K (G)$ is a C-monoid in $\mathcal F (G)$.

      \smallskip
      \item[(b)] $\mathcal B (G/K)$ is a C-monoid in $\mathcal F (G/K)$.

      \smallskip
      \item[(c)] $G/K$ is finite.
      \end{enumerate}

\smallskip
\item If $G/K$ is finite, then the following statements are equivalent:
      \begin{enumerate}
      \smallskip
      \item[(a)] $\mathcal B_K (G)$ is a seminormal C-monoid in $\mathcal F (G)$.

      \smallskip
      \item[(b)] The class semigroup of $\mathcal B_K (G)$ in $\mathcal F (G)$ is Clifford.

      \smallskip
      \item[(c)] The class semigroup of $\mathcal B (G/K)$ in $\mathcal F (G/K)$ is Clifford.

      \smallskip
      \item[(d)] $\mathcal B \big( G/K \big)$ is a seminormal C-monoid in $\mathcal F (G/K)$.

      \smallskip
      \item[(e)] $\left| \left( G/ K \right)' \right| \le 2$.
      \end{enumerate}
\end{enumerate}
\end{corollary}

\begin{proof}
1. (a) $\Leftrightarrow$ (b) follows by Theorem~\ref{thm:transfer}.2, and (b) $\Leftrightarrow$ (c) follows by Proposition~\ref{pro:B(G)}.2.

\smallskip
2. (a) $\Leftrightarrow$ (b) follows by Item 1 and \cite[Theorem 1.1]{Ge-Zh19}, (b) $\Leftrightarrow$ (c) follows by Theorem~\ref{thm:transfer}.2, and (c) $\Leftrightarrow$ (d) $\Leftrightarrow$ (e) follows by Proposition~\ref{pro:B(G)}.3.
\end{proof}

\smallskip
By Corollary~\ref{coro:seminormal}, we obtain that if $G/K$ is finite, then $\mathcal B_K (G)$ is a C-monoid, whence it is a Mori monoid by \cite[Theorem 2.9.13]{Ge-HK06}.
Moreover, by Theorem~\ref{thm:struc}.3, if $G' \subseteq K$, then $\mathcal B_K (G)$ is a Krull monoid, whence it is a Mori monoid.
The following result shows that the Mori property of $\mathcal B_K (G)$ is determined by that of $\mathcal B (G/K)$.
However, it is unknown in general when the monoid of product-one sequences is Mori, and moreover, we are not aware of any group whose monoid of product-one sequences is not Mori.

\smallskip
\begin{theorem} \label{thm:Mori}~
Let $G$ be a group with $K \unlhd G$.
Then, $\mathcal B_K (G)$ is Mori if and only if $\mathcal B (G/K)$ is Mori.
\end{theorem}

\begin{proof}
($\Leftarrow$) Suppose that $\mathcal B (G/K)$ is a Mori monoid.
Since $\mathcal F (G)$ is Krull, it follows by \cite[Proposition 2.1.11]{Ge-HK06} that $\mathcal F (G) \times \mathcal B (G/K)$ is a Mori monoid.
We consider a map $\Psi \colon \mathcal B_K (G) \to \mathcal F (G) \times \mathcal B(G/K)$ given by $\Psi (S) = \big( S, \phi_K (S) \big)$.
If $\Psi (T) \mid \Psi (S)$ in $\mathcal F (G) \times \mathcal B (G/K)$ for $S, T \in \mathcal B_K (G)$, then there exist $Z \in \mathcal F (G)$ and $W \in \mathcal B (G/K)$ such that $\Psi (S) = \Psi (T) \bdot (Z,W)$ in $\mathcal F (G) \times \mathcal B (G/K)$, and hence $\big( S, \phi_K (S) \big) = \big( T, \phi_K (T) \big) \bdot (Z,W) = \big( T \bdot Z, \phi_K (T) \bdot W \big)$.
Since $S = T \bdot Z$, it follows that
\[
  \phi_K (T) \bdot \phi_K (Z) = \phi_K (T \bdot Z) = \phi_K (S) = \phi_K (T) \bdot W \,,
\]
which implies that $\phi_K (Z) = W \in \mathcal B (G/K)$.
In view of (\ref{eq:phi}), we infer that $Z \in \mathcal B_K (G)$, and thus $T \mid S$ in $\mathcal B_K (G)$.
Hence, $\Psi$ is a divisor homomorphism, and it follows by \cite[Proposition 2.4.4]{Ge-HK06} that $\mathcal B_K (G)$ is also a Mori monoid.

($\Rightarrow$) Suppose that $\mathcal B_K (G)$ is a Mori monoid.
Let $G_0$ be a subset of $G$ consisting exactly one element from each $K$-coset.
Then, by Lemma~\ref{lem:divisor}, $\mathcal B_K (G_0)$ is a divisor-closed submonoid of $\mathcal B_K (G)$, and so the embedding $\mathcal B_K (G_0) \hookrightarrow \mathcal B_K (G)$ is a divisor homomorphism.
By \cite[Proposition 2.4.4]{Ge-HK06}, it follows that $\mathcal B_K (G_0)$ is a Mori monoid.
Note that the restriction of $\phi_K$ to $G_0$ is a bijection between $G_0$ and $G/K$, and we can extend it to an isomorphism between $\mathcal F (G_0)$ and $\mathcal F (G/K)$, also denoted by $\phi_K$.
Then, for every $S \in \mathcal F (G_0)$, we have that $S \in \mathcal B_K (G_0)$ if and only if $\phi_K (S) \in \mathcal B (G/K)$, and hence we obtain that $\phi_K \big( \mathcal B_K (G_0) \big) = \mathcal B (G/K)$.
Thus, $\mathcal B (G/K)$ is isomorphic to $\mathcal B_K (G_0)$, whence we conclude that $\mathcal B (G/K)$ is also a Mori monoid.
\end{proof}

\smallskip
In \cite[Proposition~3.6]{Ya26}, it was proved that the monoid of zero-sum sequences over a finite group with maximal elasticity is finitely generated.
Since then, it has remained an open question whether the monoid of product-$K$ sequences with maximal elasticity is finitely generated.
We answer this question affirmatively by proving that it is finitely generated.
We denote by
\[
  \mathcal B_{\rho,K} (G) = \left\{ S \in \mathcal B_K (G) \mid \rho \big( \mathsf L (S) \big) = \rho \big( \mathcal B_K (G) \big) \right\} \cup \{ 1_{\mathcal F (G)} \}
\]
the set of all product-$K$ sequences with maximal elasticity.
Suppose that $G/K$ is finite.
Then, by Theorem~\ref{thm:transfer}.1, $\mathcal{B}_{K}(G)$ has finite and accepted elasticity with $\rho \big( \mathcal{B}_{K}(G) \big) = \max \big\{ 1, \frac{\mathsf D_K (G)}{2} \big\}$, which means that $\mathcal B_{\rho,K} (G)$ is non-empty.
Moreover, since $\rho \big( \mathsf L (S \bdot T) \big) = \rho \big( \mathcal B_K (G) \big)$ for $S, T \in \mathcal B_K (G)$ with $\rho \big( \mathsf L (S) \big) = \rho \big( \mathsf L (T) \big) = \rho \big( \mathcal B_K (G) \big)$ (see \cite[Lemma 3.2]{Ya26}), it follows that $\mathcal B_{\rho,K} (G)$ forms a submonoid of $\mathcal B_K (G)$.

\smallskip
\begin{theorem} \label{thm:maxela}~
Let $G$ be a group and $K \unlhd G$.
\begin{enumerate}
\item If $G/K$ is a torsion group, then, for every subset $G_0 \subseteq G$, the following statements hold:
      \begin{enumerate}
      \smallskip
      \item The monoid $\mathcal B_K (G_0)$ is finitely generated if and only if $G_0$ is finite.

      \smallskip
      \item The monoid $\mathcal B_K (G_0)$ is locally finitely generated.
      \end{enumerate}

\smallskip
\item The monoid $\mathcal B_K (G)$ is finitely generated if and only if $G$ is finite.

\smallskip
\item If $G$ is finite, then the monoid $\mathcal B_{\rho, K} (G)$ is finitely generated.
\end{enumerate}
\end{theorem}

\begin{proof}
1. Assume that $G/K$ is a torsion group, and let $G_0 \subseteq G$ be a subset.

(a) ($\Rightarrow$) Suppose that $\mathcal B_K (G_0)$ is generated by $\{ S_1, \ldots, S_{\ell} \}$.
For each $g \in G_0$, $n_g := \ord_K (g) < \infty$ because $G/K$ is torsion.
Thus, $g^{[n_g]} \in \mathcal B_K (G_0)$, and so $g^{[n_g]}$ is a finite product of $S_1, \ldots, S_{\ell}$.
This means that $g \in \bigcup_{i \in [1,\ell]} \supp (S_i)$, so that $G_0 = \bigcup_{i \in [1,\ell]} \supp (S_i)$.
Since each $\supp (S_i)$ is finite, we conclude that $G_0$ is finite.

($\Leftarrow$) Let $G_0 = \{ g_1, \ldots, g_{\ell} \}$.
Since $G/K$ is torsion, we set $n := \lcm \{ \ord_K (g_i) \mid i \in [1,\ell] \} < \infty$.
Assume to the contrary that $\mathcal A_K (G_0)$ is infinite.
Then, there exist $x_1, \ldots, x_{\ell} \in [1,n]$ such that
\[
  M_{\mathbf x} := \{ U \in \mathcal A_K (G_0) \mid \mathsf v_{g_i} (U) \equiv x_i \pmod{n} \mbox{ for every } i \in [1,\ell] \}
\]
is infinite.
Identifying $U \in M_{\mathbf x}$ with $\big( \mathsf v_{g_1} (U), \ldots, \mathsf v_{g_{\ell}} (U) \big) \in \mathbb N^{\ell}_0$, we infer by Dickson's Lemma \cite[Theorem 1.5.3]{Ge-HK06} that $\Min (M_{\mathbf x})$ is finite, and for each $U \in M_{\mathbf x}$, there exists $V \in \Min (M_{\mathbf x})$ such that $\mathsf v_{g_i} (V) \le \mathsf v_{g_i} (U)$ for every $i \in [1,\ell]$.
Since $M_{\mathbf x}$ is infinite, we can take such $U$ and $V$ with $U \neq V$.
Then, $W := U \bdot V^{[-1]} = {\small \prod}^{\bullet}_{i \in [1,\ell]} g^{[\mathsf v_{g_i} (U) - \mathsf v_{g_i} (V)]}_i \in \mathcal F (G_0)$, and since $U, V \in M_{\mathbf x}$, it follows that $\mathsf v_{g_i} (U) \equiv x_i \equiv \mathsf v_{g_i} (V) \pmod{n}$ for every $i \in [1,\ell]$.
Thus, for every $i \in [1,\ell]$, $\mathsf v_{g_i} (U) - \mathsf v_{g_i} (V) = n m_i$ for some $m_i \in \mathbb N_0$, and since $g^{[n]}_i \in \mathcal B_K (G_0)$ for every $i \in [1,\ell]$, we obtain that
\[
  W = {\small \prod}^{\bullet}_{i \in [1,\ell]} g^{[\mathsf v_{g_i} (U) - \mathsf v_{g_i} (V)]}_i = {\small \prod}^{\bullet}_{i \in [1,\ell]} g^{[n m_i]}_i = {\small \prod}^{\bullet}_{i \in [1,\ell]} \left( g^{[n]}_i \right)^{[m_i]} \in \mathcal B_K (G_0) \,.
\]
Hence, $W$ is non-trivial and $U = W \bdot V$, contradicting that $U \in \mathcal A_K (G_0)$.
Therefore, $\mathcal A_K (G_0)$ is finite, and so we conclude that $\mathcal B_K (G_0)$ is finitely generated.

(b) Let $S \in \mathcal B_K (G_0)$.
Then, $[[S]]$ is a divisor-closed submonoid of $\mathcal B_K (G_0)$, and as in the proof of Lemma~\ref{lem:divisor}, $[[S]] = \mathcal B_K (G_1)$, where $G_1 = \bigcup_{T \in [[S]]} \supp (T) \subseteq G_0$.
By definition, $G_1 \subseteq \supp (S)$, and since $S \in \mathcal B_K (G_1)$, it follows that $\supp (S) \subseteq G_1$, whence $G_1 = \supp (S)$ and $[[S]] = \mathcal B_K ( \supp (S) )$.
Since $\supp (S)$ is finite, we infer by (a) that $[[S]] = \mathcal B_K (\supp (S))$ is finitely generated.
Thus, $\mathcal B_K (G_0)$ is a locally finitely generated monoid.

\smallskip
2. ($\Leftarrow$) If $G$ is finite, then $G/K$ is also finite, and thus the assertion follows by Item 1.

($\Rightarrow$) Suppose that the monoid $\mathcal B_K (G)$ is finitely generated.
Then, the set $\mathcal A_K (G)$ of atoms is finite.
Since every element in $K$ can be viewed as a minimal product-$K$ sequence of length 1, it follows that $K \subseteq \mathcal A_K (G)$, and thus $K$ is also finite.
If $g \in G \setminus K$, then $g^{-1} \in G \setminus K$, and hence $g \bdot g^{-1} \in \mathcal A_K (G)$.
If $G/K$ is infinite, then we can take infinitely many such minimal product-$K$ sequences, implying that $\mathcal A_K (G)$ contains infinitely many elements, a contradiction.
Thus, $G/K$ is also finite, and it follows by Lagrange's Theorem that $G$ is finite.

\smallskip
3. If $K = G$, then $\mathcal B_{\rho, K} (G) = \mathcal F (G)$, and since $G$ is finite, $\mathcal F (G)$ is finitely generated.
Suppose now that $K \neq G$, so that $\rho \big( \mathcal B_K (G) \big) = \frac{\mathsf D_K (G)}{2}$.
Then, by Item 2, $\mathcal B_{K} (G)$ is finitely generated, so we may assume that $\mathcal A_K (G) = \{ U_1, \ldots, U_m \}$.
Let $\pi \colon \mathbb N_0^m \to \mathcal B_K (G)$ be the factorization homomorphism defined by $\pi (a_1, \ldots, a_m) = U_1^{[a_1]} \bdot \ldots \bdot U_m^{[a_m]}$.
For $\mathbf x = (x_1, \ldots, x_m) \in \mathbb N_0^m$, put $|\mathbf x| = x_1 + \cdots + x_m$.
Now we define
\[
  \mathcal R_\rho = \big\{ (\mathbf x, \mathbf y) \in \mathbb N_0^m \times \mathbb N_0^m \mid \pi (\mathbf x) = \pi (\mathbf y) \,\, \mbox{ and } \,\, 2 |\mathbf x| = \mathsf D_K (G) | \mathbf y | \big\} \subseteq \mathbb N^{m}_0 \times \mathbb N^{m}_0 \,,
\]
and we assert that $\mathcal R_{\rho}$ is a finitely generated submonoid of $\mathbb N^{m}_0 \times \mathbb N^{m}_0$.
Let $( \mathbf x, \mathbf y), ( \mathbf u, \mathbf v) \in \mathcal R_{\rho}$.
Since $\pi$ is a homomorphism, we obtain that
\begin{itemize}
\item[(i)] $\pi (\mathbf x + \mathbf u) = \pi (\mathbf x) \bdot \pi (\mathbf u) = \pi (\mathbf y) \bdot \pi (\mathbf v) = \pi (\mathbf y + \mathbf v)$, and

\smallskip
\item[(ii)] $2 | \mathbf x + \mathbf u | = 2 | \mathbf x | + 2 | \mathbf u | = \mathsf D_K (G) | \mathbf y | + \mathsf D_K (G) | \mathbf v | = \mathsf D_K (G) | \mathbf y + \mathbf v |$,
\end{itemize}
and thus $( \mathbf x + \mathbf u, \mathbf y + \mathbf v) \in \mathcal R_{\rho}$, which ensures that $\mathcal R_{\rho}$ is a monoid.
If $( \mathbf x, \mathbf y)$ divides $(\mathbf u, \mathbf v)$ in $\mathbb N^{m}_0 \times \mathbb N^{m}_0$, then $( \mathbf u - \mathbf x, \mathbf v - \mathbf y ) \in \mathbb N^{m}_0 \times \mathbb N^{m}_0$.
Since
\[
  \pi (\mathbf u - \mathbf x ) \bdot \pi (\mathbf y) = \pi (\mathbf u - \mathbf x ) \bdot \pi (\mathbf x) = \pi (\mathbf u)  = \pi (\mathbf v) = \pi (\mathbf v - \mathbf y) \bdot \pi (\mathbf y) \,,
\]
we infer that $\pi (\mathbf u - \mathbf x) = \pi (\mathbf v - \mathbf y)$.
Moreover,
\[
  2 | \mathbf u - \mathbf x | + 2 | \mathbf x | = 2 | \mathbf u | = \mathsf D_K (G) | \mathbf v| = \mathsf D_K (G) | \mathbf v - \mathbf y| + \mathsf D_K (G) |\mathbf y| \,,
\]
and since $2 | \mathbf x| = \mathsf D_K (G) | \mathbf y|$, we obtain that $2 | \mathbf u - \mathbf x | = \mathsf D_K (G) | \mathbf v - \mathbf y|$.
Thus, $( \mathbf u - \mathbf x, \mathbf v - \mathbf y ) \in \mathcal R_{\rho}$, and hence we infer that the embedding $\mathcal R_{\rho} \hookrightarrow \mathbb N^{m}_0 \times \mathbb N^{m}_0$ is a divisor homomorphism.
Since $\mathbb N^{m}_0 \times \mathbb N^{m}_0$ is finitely generated, it follows by \cite[Proposition 2.7.5]{Ge-HK06} that $\mathcal R_{\rho}$ is also a finitely generated submonoid.
Now, we consider the map $\overline{\theta} \colon \mathcal R_{\rho} \overset{\theta}{\to} \mathbb N^{m}_0 \overset{\pi}{\to} \mathcal B_K (G)$, where $\theta$ is the projection onto the first coordinate.
Then, it suffices to show that $\overline{\theta} \big( \mathcal R_{\rho} \big) = \mathcal B_{\rho, K} (G)$.

First, let $S \in \overline{\theta} \big( \mathcal R_{\rho} \big)$.
Then, there exists $(\mathbf x, \mathbf y) \in \mathcal R_{\rho}$ with $\overline{\theta} (\mathbf x, \mathbf y) = S$.
If $(\mathbf x, \mathbf y)$ is the identity in $\mathcal R_{\rho}$, then $S = 1_{\mathcal F (G)}$ is a trivial sequence, and so $S \in \mathcal B_{\rho,K} (G)$.
If $(\mathbf x, \mathbf y)$ is non-identity in $\mathcal R_{\rho}$, then both $|\mathbf x|$ and $|\mathbf y|$ are non-zero, and so $\frac{|\mathbf x|}{|\mathbf y|} = \frac{\mathsf D_K(G)}{2} = \rho \big( \mathcal{B}_{K}(G) \big)$.
This means that $S$ has two factorizations whose lengths realize the ratio $\rho \big( \mathcal{B}_{K}(G) \big)$, so that $\rho \big( \mathsf L (S) \big) \geq \rho \big( \mathcal{B}_{K}(G) \big)$.
In view of (\ref{eq:rho}), we obtain that $\rho \big( \mathsf L (S) \big) = \rho \big( \mathcal{B}_{K}(G) \big)$, so that $S \in \mathcal B_{\rho,K} (G)$.

Conversely, let $S\in \mathcal B_{\rho,K}(G)$.
If $S = 1_{\mathcal F (G)}$ is trivial, then $S$ is the image of the identity in $\mathcal R_{\rho}$, and so $S \in \overline{\theta} \big( \mathcal R_{\rho} \big)$.
Suppose that $S$ is non-trivial.
Then, since $\rho \big( \mathsf L (S) \big) = \rho \big( \mathcal{B}_{K}(G) \big) = \frac{\mathsf D_K(G)}{2}$, there exist two factorizations $z$ and $z'$ of $S$ with $\frac{|z|}{|z'|} = \frac{\mathsf D_K (G)}{2}$.
Since $U_1, \ldots, U_m$ are all atoms in $\mathcal B_K (G)$, we can write
\[
  z := U^{[x_1]}_1 \bdot \ldots \bdot U^{[x_m]}_m \quad \mbox{ and } \quad z' := U^{[y_1]}_1 \bdot \ldots \bdot U^{[y_m]}_m
\]
with $x_1,\ldots, x_m, y_1, \ldots, y_m \in \mathbb N_0$.
If we set $\mathbf x = (x_1, \ldots, x_m)$ and $\mathbf y = (y_1, \ldots, y_m)$, then $|z| = | \mathbf x|$, $|z'| = | \mathbf y|$, $\pi ( \mathbf x ) = \pi ( \mathbf y )$ and $2 |\mathbf x| = \mathsf D_K (G) | \mathbf y|$.
Thus, $(\mathbf x, \mathbf y) \in \mathcal R_{\rho}$ with $\overline{\theta} (\mathbf x, \mathbf y) = S$, whence $S \in \overline{\theta} \big( \mathcal R_{\rho} \big)$.
\end{proof}

\smallskip
\begin{remark} \label{rmk:torsion}~
Theorem~\ref{thm:maxela}.1 does not necessarily hold without the torsion condition. To see this, let $G = \langle \alpha, \tau \mid \tau^{2} = 1_G \mbox{ and } \tau \alpha = \alpha^{-1}\tau \rangle$ be an infinite dihedral group, $K = \{ 1_G \}$, and $G_0 = \{ \alpha, \tau \}$.
For each $n \in \mathbb N_0$, we set $U_n = \alpha^{[2n]} \bdot \tau^{[2]} \in \mathcal F (G_0)$, and then it is easy to obtain that $\mathcal A_K (G_0) = \{ U_n \mid n \in \mathbb N_0 \}$.
Thus, $\mathcal B_K (G_0)$ is not finitely generated.
Moreover, if we set $S = U_1 \in \mathcal B_K (G_0)$, then, for every $n \in \mathbb N$, $S^{[n]} = U_n \bdot \big( \tau^{[2]} \big)^{[n-1]}$, and since $\tau^{[2]} \in \mathcal B_K (G_0)$, it follows that $U_n \mid S^{[n]}$ in $\mathcal B_K (G_0)$.
This means that $\mathcal A_K (G_0) \setminus \{ U_0 \} \subseteq [[S]]$, implying that $\mathcal B_K (G_0)$ is not locally finitely generated.
\end{remark}

\smallskip
Finally, we turn to the Isomorphism Problem for monoids of product-$K$ sequences.
For trivial $K$, Geroldinger and Oh \cite{Ge-Oh25} recently proved that the monoid structure determines the structure of the underlying torsion group, up to isomorphism.
The corresponding problem still remains open for arbitrary groups.
For non-trivial $K$, one cannot in general expect an isomorphism between $\mathcal B_{K_1} (G_1)$ and $\mathcal B_{K_2} (G_2)$ to imply $G_1 \cong G_2$ and $K_1 \cong K_2$.
Nevertheless, the following result provides a partial analogue of the product-one case: apart from a small factorial exception, the monoid structure determines the structure of the torsion quotient group, up to isomorphism, and the cardinality of the normal subgroup.

\smallskip
\begin{theorem} \label{thm:iso}~
For $i \in [1,2]$, let $G_i$ be a group and $K_i \unlhd G_i$ be such that $G_i / K_i$ is a torsion group.
Then, the monoids $\mathcal B_{K_1} (G_1)$ and $\mathcal B_{K_2} (G_2)$ are isomorphic if and only if one of the following statements holds:
\begin{enumerate}
\smallskip
\item[(a)] The groups $G_1 / K_1$ and $G_2 / K_2$ are isomorphic, and $K_1$ and $K_2$ have the same cardinality.

\smallskip
\item[(b)] The groups $G_1$ and $G_2$ each have order 2.
\end{enumerate}
\end{theorem}

\begin{proof}
($\Leftarrow$) Suppose first that $G_1$ and $G_2$ each have order 2.
Then, $G_1 \cong G_2 \cong C_2$ is a cyclic group of order 2, and for each $i \in [1,2]$, it follows by Theorem~\ref{thm:transfer}.4 that $\mathcal B_{K_i} (G_i)$ is factorial.
For $G = \langle g \rangle$ with $\ord (g) = 2$, we observe that
\[
  \mathcal B_{\{ 1_G\}} (G) = \mathcal F \big( \{ 1_G, g^{[2]} \} \big) \cong \mathbb N^{2}_0 \cong \mathcal F (G) = \mathcal B_G (G) \,.
\]
Thus, we conclude that $\mathcal B_{K_1} (G_1)$ and $\mathcal B_{K_2} (G_2)$ are isomorphic, regardless of the specific choice of $K_i$.
Hence, we may assume that one of $G_i$ is not of order 2, say $G_1$.

Suppose now that $G_1 / K_1$ and $G_2 / K_2$ are isomorphic, and $K_1$ and $K_2$ have the same cardinality.
Assume to the contrary that $G_2 \cong C_2$.
If $K_2 = \{ 1_{G_2} \}$, then $|K_1| = |K_2| = 1$ and $G_1 / K_1 \cong G_2 / K_2 \cong C_2$, whence $G_1 \cong C_2$, a contradiction.
If $K_2 = G_2$, then $|K_1| = |K_2| = 2$ and $G_1 / K_1 \cong G_2 / K_2$ is trivial, whence $G_1 = K_1 \cong C_2$, again a contradiction.
Thus, we further assume that $G_2$ also is not of order 2.
Now, let $\varphi \colon G_1 / K_1 \to G_2 / K_2$ be an isomorphism.
Note that, for each $i \in [1,2]$, the $K_i$-cosets form a partition of $G_i$.
Let $G_1 = \underset{\alpha \in \Lambda}{\mathbin{\dot{\bigcup}}} g_{\alpha}K_1$ and $G_2 = \underset{\beta \in \Gamma}{\mathbin{\dot{\bigcup}}} h_{\beta}K_2$.
Then, for $\alpha \in \Lambda$, $\varphi (g_{\alpha}K_1) = h_{\beta} K_2$ for some $\beta \in \Gamma$, and since $|K_1| = |K_2|$, it follows that $|g_{\alpha}K_1| = |K_1| = |K_2| = |h_{\beta}K_2| = |\varphi (g_{\alpha}K_1)|$.
This allows us to take a bijection $\varphi_{\alpha}$ between $g_{\alpha}K_1$ and $\varphi (g_{\alpha}K_1)$.
We define a map
\[
  \overline{\varphi} \colon G_1 \to G_2 \,\, \mbox{ by } \,\, \overline{\varphi} (g) = \varphi_{\alpha} (g) \,\, \mbox{ for } \,\, g \in g_{\alpha}K_1 \,.
\]
Clearly, $\overline{\varphi}$ is well-defined, and we assert that $\overline{\varphi}$ is bijective.

For injectivity, let $\overline{\varphi} (g_1) = \overline{\varphi} (g_2)$ for $g_1, g_2 \in G_1$.
Then, $\overline{\varphi} (g_1) = \varphi_{\alpha_1} (g_1) \in \varphi (g_{\alpha_1}K_1)$ and $\overline{\varphi} (g_2) = \varphi_{\alpha_2} (g_2) \in \varphi (g_{\alpha_2}K_1)$ for some $\alpha_1, \alpha_2 \in \Lambda$, which implies that $\varphi (g_{\alpha_1}K_1) \cap \varphi (g_{\alpha_2}K_1) \neq \emptyset$.
Thus, we obtain that $\varphi (g_{\alpha_1}K_1) = \varphi (g_{\alpha_2}K_1)$, and so we infer that $\alpha_1 = \alpha_2$.
Since $\varphi_{\alpha_1}$ is injective, it follows that $g_1 = g_2$.

For surjectivity, let $h \in G_2$.
Then, $h \in h_{\beta}K_2$ for some $\beta \in \Gamma$, and under the isomorphism $\varphi$, there exists $\alpha \in \Lambda$ such that $\varphi (g_{\alpha}K_1) = h_{\beta}K_2$.
Since $\varphi_{\alpha}$ is surjective, there exists $g \in g_{\alpha}K_1$ such that $\varphi_{\alpha} (g) = h \in \varphi (g_{\alpha}K_1) = h_{\beta}K_2$.
This means that $g \in G_1$ and $\overline{\varphi} (g) = \varphi_{\alpha} (g) = h$.

Moreover, for every $g \in G_1$, we obtain that
\begin{equation} \label{eq:compo}~
  \phi_{K_2} \big( \overline{\varphi} (g) \big) = \phi_{K_2} \big( \varphi_{\alpha} (g) \big) = \varphi (g_{\alpha}K_1) = \varphi \big( \phi_{K_1} (g) \big) \,.
\end{equation}
Then, we can extend $\overline{\varphi}$ to a monoid isomorphism between free abelian monoids $\mathcal F (G_1)$ and $\mathcal F (G_2)$, and by abuse of notation, we also denote this extension by $\overline{\varphi}$.

Since $\varphi$ is an isomorphism, in views of (\ref{eq:phi}) and \ref{eq:compo}, we obtain that, for every $S \in \mathcal F (G_1)$, $S \in \mathcal B_{K_1} (G_1)$ if and only if $\phi_{K_1} (S) \in \mathcal B (G_1 / K_1)$ if and only if $\varphi \big( \phi_{K_1} (S) \big) \in \mathcal B (G_2 / K_2)$ if and only if $\phi_{K_2} \big( \overline{\varphi} (S) \big) = \varphi \big( \phi_{K_1} (S) \big) \in \mathcal B (G_2 / K_2)$ if and only if $\overline{\varphi} (S) \in \mathcal B_{K_2} (G_2)$.
Thus, $\overline{\varphi} \big( \mathcal B_{K_1} (G_1) \big) = \mathcal B_{K_2} (G_2)$, and therefore the monoids $\mathcal B_{K_1} (G_1)$ and $\mathcal B_{K_2} (G_2)$ are isomorphic.

($\Rightarrow$) Suppose that the monoids $\mathcal B_{K_1} (G_1)$ and $\mathcal B_{K_2} (G_2)$ are isomorphic.

\smallskip
\noindent
{\bf CASE 1.} The monoids $\mathcal B_{K_i} (G_i)$ are factorial.

\smallskip
By Theorem~\ref{thm:transfer}.4, for each $i \in [1,2]$, we have that either $K_i = G_i$, or $G_i \cong C_2$ and $K_i$ is trivial.

Suppose first that $K_i = G_i$ for some $i \in [1,2]$.
By symmetry, we may assume that $K_1 = G_1$.
Then, $\mathcal B_{K_2} (G_2) \cong \mathcal B_{K_1} (G_1) = \mathcal F (G_1)$.
If $K_2 = G_2$, then $\mathcal F (G_2) = \mathcal B_{K_2} (G_2) \cong \mathcal F (G_1)$, and thus we infer that $|K_1| = |G_1| = |G_2| = |K_2|$ and $G_1 / K_1 \cong G_2 / K_2$ is trivial, which satisfies (a).
If $K_2 \neq G_2$, then $G_2 \cong C_2$ and $K_2$ is trivial, and hence $\mathcal F (G_1) = \mathcal B_{K_1} (G_1) \cong \mathcal B_{K_2} (G_2) \cong \mathcal F (C_2)$, whence $G_1 \cong G_2 \cong C_2$, which satisfies (b).

Finally, suppose that $K_i \neq G_i$ for all $i \in [1,2]$.
Then, $G_1 \cong G_2 \cong C_2$ and both $K_1$ and $K_2$ are trivial, which directly satisfies both (a) and (b).

\smallskip
\noindent
{\bf CASE 2.} The monoids $\mathcal B_{K_i} (G_i)$ are not factorial.

\smallskip
Then, by Theorem~\ref{thm:transfer}.4, we have that $|G_i| \ge 3$ and $G_i \neq K_i$ for every $i \in [1,2]$.
Since $\mathcal B_{K_1} (G_1)$ and $\mathcal B_{K_2} (G_2)$ are isomorphic, we can extend it to an isomorphism between their quotient groups $\mathsf q \big( \mathcal B_{K_1} (G_1) \big)$ and $\mathsf q \big( \mathcal B_{K_2} (G_2) \big)$, and also an isomorphism between their complete integral closures.
Let $\varphi \colon \widehat{\mathcal B_{K_1} (G_1)} \to \widehat{\mathcal B_{K_2} (G_2)}$ be an isomorphism obtained in this way.
Since each $G_i / K_i$ is a torsion group, it follows by Theorem~\ref{thm:struc}.1 that the embeddings $\widehat{\mathcal B_{K_i} (G_i)} \hookrightarrow \mathcal F (G_i)$ are divisor theories.
By the Uniqueness Theorem for divisor theories \cite[Theorem 2.4.7]{Ge-HK06}, there exists an isomorphism $\overline{\varphi} \colon \mathcal F (G_1) \to \mathcal F (G_2)$ satisfying $\overline{\varphi} \, \big\vert_{\widehat{\mathcal B_{K_1} (G_1)}} = \varphi$.
Since every monoid isomorphism between free abelian monoids stems from a bijection between their bases, we have a bijection between $G_1$ and $G_2$, and by abuse of notation, we also denote this bijection by $\overline{\varphi}$.
Since $\overline{\varphi}$ stems from an isomorphism between $\mathcal B_{K_1} (G_1)$ and $\mathcal B_{K_2} (G_2)$, we have that
\begin{equation} \label{eq:preserve}~
  S \in \mathcal B_{K_1} (G_1) \quad \mbox{ if and only if } \quad \overline{\varphi} (S) \in \mathcal B_{K_2} (G_2) \,.
\end{equation}
Since every element in $K_i$ is an atom (of length 1) in $\mathcal B_{K_i} (G_i)$, it follows that, for every $g \in G_1$, $g \in K_1$ if and only if $g \in \mathcal B_{K_1} (G_1)$ if and only if $\overline{\varphi} (g) \in \mathcal B_{K_2} (G_2)$ if and only if $\overline{\varphi} (g) \in K_2$, which implies that $\overline{\varphi} (K_1) = K_2$, whence $K_1$ and $K_2$ have the same cardinality.

For each $i \in [1,2]$, we now define a relation on $G_i$ as follows: for $g, h \in G_i$, $g \sim_i h$ if and only if, for every $t \in G_i$,
\[
  g \bdot t \in \mathcal B_{K_i} (G_i) \quad \Leftrightarrow \quad h \bdot t \in \mathcal B_{K_i} (G_i) \,.
\]
Clearly, $\sim_i$ is an equivalence relation on $G_i$, and we denote by $[g]_{\sim_i}$ the equivalence class of $g \in G_i$.
Then, we assert that
\begin{equation} \label{eq:equiv}~
  [g]_{\sim_i} = gK_i \,\, \mbox{ for every } \,\, g \in G_i \,.
\end{equation}
If $x \in [g]_{\sim_i}$, then, since $g \bdot g^{-1} \in \mathcal B_{K_i} (G_i)$, it follows that $x \bdot g^{-1} \in \mathcal B_{K_i} (G_i)$, and by Lemma~\ref{lem:shift}, we obtain that $g^{-1}x \in K_i$, equivalently $x \in g K_i$.
Let $x \in gK_i$ and $t \in G_i$.
If $x \bdot t \in \mathcal B_{K_i} (G_i)$, then by Lemma~\ref{lem:shift}, $tx \in K_i$, and so $x \in t^{-1} K_i \cap gK_i$, which implies that $t^{-1}K_i = gK_i$.
Thus, $tg \in K_i$, and hence $g \bdot t \in \mathcal B_{K_i} (G_i)$.
If $g \bdot t \in \mathcal B_{K_i} (G_i)$, then by Lemma~\ref{lem:shift}, $tg \in K_i$, equivalently $gK_i = t^{-1}K_i$.
Since $x \in gK_i$, it follows that $x \in t^{-1}K_i$, so that $tx \in K_i$, whence $x \bdot t \in \mathcal B_{K_i} (G_i)$.
Thus, the assertion follows.

Since $\overline{\varphi}$ is bijective, in view of (\ref{eq:preserve}), we directly obtain that $g \sim_1 h$ if and only if $\overline{\varphi} (g) \sim_2 \overline{\varphi} (h)$ for every $g, h \in G_1$.
In view of (\ref{eq:equiv}), we have that, for every $g, h \in G_1$,
\[
  gK_1 = hK_1 \quad \mbox{ if and only if } \quad \overline{\varphi} (g) K_2 = \overline{\varphi} (h) K_2 \,,
\]
which ensures that a map $\psi \colon G_1 / K_1 \to G_2 / K_2$, defined by $\psi (gK_1) = \overline{\varphi} (g) K_2$, is a well-defined bijection.
Using the same notation, we obtain an isomorphism $\psi \colon \mathcal F (G_1 / K_1) \to \mathcal F (G_2 / K_2)$ satisfying
\begin{equation} \label{eq:comm2}~
  \psi \big( \phi_{K_1} (S) \big) = \phi_{K_2} \big( \overline{\varphi} (S) \big) \quad \mbox{ for every } \,\, S \in \mathcal F (G_1) \,.
\end{equation}
Let $T \in \mathcal F (G_1 / K_1)$, and take $S \in \mathcal F (G_1)$ such that $\phi_{K_1} (S) = T$.
Then, in views of (\ref{eq:phi}), (\ref{eq:preserve}), and (\ref{eq:comm2}), we obtain that
$T \in \mathcal B (G_1 / K_1)$ if and only if $S \in \mathcal B_{K_1} (G_1)$ if and only if $\overline{\varphi} (S) \in \mathcal B_{K_2} (G_2)$ if and only if $\psi (T) = \psi \big( \phi_{K_1} (S) \big) = \phi_{K_2} \big( \overline{\varphi} (S) \big) \in \mathcal B (G_2 / K_2)$, which ensures that the restriction $\psi \, \big\vert_{\mathcal B (G_1 / K_1)}$ is an isomorphism between the monoids of product-one sequences $\mathcal B (G_1 / K_1)$ and $\mathcal B (G_2 / K_2)$.
Since $G_i / K_i$ are both torsion groups, it follows by \cite[Theorem 1.1]{Ge-Oh25} that the groups $G_1 / K_1$ and $G_2 / K_2$ are isomorphic, whence (a) holds.
\end{proof}

\medskip
\noindent
{\bf Conflict of interest.} On behalf of all authors, the corresponding author states that there is no conflict of interest.

\medskip
\noindent
{\bf Acknowledgement.} The authors would like to thank the anonymous referees for their time and effort in reading this paper carefully.
They provided valuable comments and suggestions, which greatly helped to improve the original version of this paper.
%This paper was completed during J.S. Oh's research visit to the AlgNTh group at University of Graz.
%The authors are grateful to Alfred Geroldinger for his helpful comments on a preliminary version of this paper, for the excellent working conditions, and for all his hospitality.

% % % % % % % % % % % % % % % % % % % % % % % % % % % % % % % % % % % % % % % % % % % % % % % % % % % % % %
% % %             % % % % % % %               % % % % % % %               % % % % % % %               % % %
% % %             % % % % % % %               % % % % % % %               % % % % % % %               % % %
% % % % % % % % % % % % % % % % % % % % % % % % % % % % % % % % % % % % % % % % % % % % % % % % % % % % % %

\medskip
\providecommand{\bysame}{\leavevmode\hbox to3em{\hrulefill}\thinspace}
\providecommand{\MR}{\relax\ifhmode\unskip\space\fi MR }
\providecommand{\MRhref}[2]{%
  \href{http://www.ams.org/mathscinet-getitem?mr=#1}{#2}
}
\providecommand{\href}[2]{#2}

\medskip

\end{document}